\RequirePackage{fix-cm}

\documentclass{svjour3}                                  

\smartqed  

\usepackage{mathptmx}
\usepackage{amsmath} 
\usepackage{amscd,amssymb,latexsym,graphicx}
\usepackage{amscd}
\usepackage[all]{xy}
\usepackage{cite}
\usepackage{hycolor}
\usepackage{xcolor}
\usepackage[
                     colorlinks=false,
                     linkbordercolor={1 0 0},   
                     citebordercolor={0 1 0},   
                     urlbordercolor ={0 1 1},    
                     ]{hyperref}
\newtheorem{hypothesis}[definition]{Hypothesis}
\renewcommand{\1}{{{\mathchoice {\rm 1\mskip-4mu l} {\rm 1\mskip-4mu l}
{\rm 1\mskip-4.5mu l} {\rm 1\mskip-5mu l}}}}
\newcommand{\C}{{\mathbb{C}}}

\newcommand{\N}{{\mathbb{N}}}

\newcommand{\R}{{\mathbb{R}}}

\newcommand{\Z}{{\mathbb{Z}}}
\newcommand{\Bb}{{\mathcal{B}}}
\newcommand{\Dd}{{\mathcal{D}}}

\newcommand{\Gg}{{\mathcal{G}}}   

\newcommand{\Ll}{{\mathcal{L}}}   

\newcommand{\Oo}{{\mathcal{O}}}

\newcommand{\Ss}{{\mathcal{S}}}

\newcommand{\Uu}{{\mathcal{U}}}
\newcommand{\Vv}{{\mathcal{V}}}

\newcommand{\range}{{\rm range\, }}  

\newcommand{\IND}{{\rm ind}}       

\newcommand{\grad}{{\rm grad }}    
\newcommand{\Ho}{{\rm H}}             

\newcommand{\norm}{{\rm norm}}

\newcommand{\eps}{{\varepsilon}}

\renewcommand{\tt}{{\mathfrak t}}  

\def\NABLA#1{{\mathop{\nabla\kern-.5ex\lower1ex\hbox{$#1$}}}}
\def\Nabla#1{\nabla\kern-.5ex{}_{#1}}
\def\Tabla#1{\Tilde\nabla\kern-.5ex{}_{#1}}
\def\abs#1{\mathopen|#1\mathclose|}   
\def\Abs#1{\left|#1\right|}            
\def\norm#1{\mathopen\|#1\mathclose\|}
\def\Norm#1{\left\|#1\right\|}

\renewcommand{\Tilde}{\widetilde}

\newcommand{\p}{{\partial}}

\begin{document}

\title{A backward $\lambda$-Lemma for the forward heat
  flow\thanks{Research supported by Universit\"at Bielefeld and
    Funda\c{c}\~{a}o de Amparo \`{a} Pesquisa do Estado de S\~{a}o
    Paulo, FAPESP grants 2011/01830-1 and
    2013/20912-4}}

\subtitle{Dedicated to the memory of V.I.~Arnol'd}

\titlerunning{Backward $\lambda$-Lemma}

\author{Joa Weber 
}


\institute{Joa Weber \at
                IMECC UNICAMP,
                Rua S\'{e}rgio Buarque de Holanda 651,
                CEP 13083-859, Campinas-SP, Brasil
              \\
              \email{joa@math.sunysb.edu}           
}

\date{Submitted: 14 December 2012 
         / Revised: 07 February 2014 / Accepted: February 2014}

\maketitle

\begin{abstract}
The inclination or $\lambda$-Lemma is a
fundamental tool in finite dimensional
hyperbolic dynamics.
In contrast to finite dimension, we consider the forward
semi-flow on the loop space of a closed Riemannian manifold
$M$ provided by the heat flow.
The main result is a backward
$\lambda$-Lemma for the heat flow near a hyperbolic
fixed point $x$. There are the following novelties.
Firstly, infinite versus finite dimension.
Secondly, semi-flow versus flow.
Thirdly, suitable adaption provides a new proof
in the finite dimensional case.
Fourthly and a~priori most surprisingly,
our $\lambda$-Lemma moves the given disk transversal
to the unstable manifold backward in time, although
there is no backward flow.
As a first application we propose
a new method to calculate
the Conley homotopy~index~of~$x$.


\end{abstract}

\section{Introduction and main results}
Assume $M$ is a closed smooth manifold of dimension
$n\ge 1$ equipped
with a Riemannian metric $g$ and the Levi-Civita
connection $\nabla$. Throughout smooth means
$C^\infty$ smooth.
The {\bf loop space} is the Hilbert manifold
$\Lambda M:=W^{1,2}(S^1,M)$ of absolutely continuous
loops in $M$ with square integrable derivative.
We identify
$S^1=\R/\Z$ and think of $\gamma\in\Lambda M$
as a map $\gamma:\R\to M$ that satisfies
$\gamma(t+1)=\gamma(t)$. Pick a smooth function
$V:S^1\times M\to \R$ and set $V_t(q):=V(t,q)$.

Consider the {\bf heat equation}
\begin{equation}\label{eq:heat}
   \p_su - \Nabla{t}\p_tu - \grad V_t(u) = 0
\end{equation}
for smooth maps
$\R\times S^1\to M:(s,t)\mapsto u(s,t)$.
It is well known that the corresponding Cauchy problem
for the map $[0,\infty)\to\Lambda M:s\mapsto u_s:=u(s,\cdot)$
admits a unique solution.
The associated forward semi-flow $\varphi$ on $\Lambda M$
is called the {\bf heat flow}.
It is a continuous map
\begin{equation*}
     \varphi:[0,\infty)\times\Lambda M\to\Lambda M
\end{equation*}
which is of class $C^1$ on $(0,\infty)$;
see~(\ref{eq:cauchy-local}) for its
representative $\phi$ in local coordinates.
For $\gamma\in\Lambda M$ abbreviate
$\dot\gamma=\frac{d}{dt}\gamma$.
Because~(\ref{eq:heat}) is the downward
$L^2$ gradient equation of the {\bf action functional}
$\Ss_V:\Lambda M\to\R$ given by
$$
     \Ss_V(\gamma)=\int_0^1\left(\frac12\abs{\dot\gamma (t)}^2
     -V_t(\gamma(t))\right) dt,
$$
the fixed points of the heat flow are
the critical points of the action. The latter
are (perturbed) closed geodesics, that is solutions
$x:S^1\to M$ to the second order ODE
$$-\Nabla{t}\dot x-\nabla V_t(x)=0.$$
By $\IND(x)$ we denote the Morse index of $x$.
{\bf Nondegeneracy} of the critical point $x$,
that is nondegeneracy of the Hessian of $\Ss_V$ at $x$,
corresponds to hyperbolicity of the fixed point $x$.

Fix a nondegenerate critical point $x$ of the action
$\Ss_V$. While the
{\bf\boldmath stable manifold $W^s(x)$} is
defined in the usual way as the set of all points which
flow in forward time asymptotically into $x$,
it is of infinite dimension. The fact that $W^s(x)$ is
\emph{globally} embedded
is, firstly, remarkable since the standard method
of pulling back coordinates near $x$ using the
backward flow is obviously not available. Secondly,
except for~\cite[Thm.~6.1.9]{Henry-81-GeomTheory}
this fact is usually not mentioned at all in the literature---unlike
the widely known \emph{local} submanifold property near
$x$; see Section~\ref{sec:loc-stab-mf}.
In contrast, without a backward flow the definition of the
{\bf\boldmath unstable manifold $W^u(x)$}
becomes somewhat awkward: It is the set of
endpoints of all heat flow trajectories parametrized by
$(-\infty,0]$ and emanating at time $-\infty$ from $x$.
On the other hand, this definition lends itself
to define a backward flow on $W^u(x)$ which immediately
implies that the unstable manifold is globally embedded.
Most importantly, its dimension given by $\IND(x)$ is finite;
see e.g.~\cite{Joa-HEATMORSE-II}.
\emph{It is this finite dimensionality which is one of two pillars on which this paper is based.}
A key consequence is smoothness of every
$\gamma\in W^u(x)$; see Remark~\ref{rem:mixed-Cauchy}.

\subsection{Some history}
In finite dimensional hyperbolic dynamics
there are two fundamental tools:
The Grobman-Hartman
Theorem~\cite{Grobman-59,Hartman-60} and 
the $\lambda$-Lemma~\cite{Palis-MS}.
While the first is powerful concerning topological
questions the latter reigns in the differentiable world.
It even implies the former.
The {\bf\boldmath$\lambda$-Lemma}
asserts that the backward flow applied
to any disk $\Dd$ transversal to
the unstable manifold and of complementary dimension
converges in $C^1$ to the local stable manifold, see
Figure~\ref{fig:fig-stable-fibr-global},
and similarly for the forward flow.
For a beautiful presentation
see~\cite{Palis-deMelo}.
Since convergence is in $C^1$, the $\lambda$-Lemma
is also called {\bf inclination lemma}. 

\emph{The second pillar on which this paper is based is the replacement of the absent backward flow on the loop space by the family of preimages $s\mapsto{\varphi_s}^{-1}\Dd$.}
This idea was born when we attempted 
to construct a Morse filtration of the loop space
using the method of Abbondandolo and 
Majer~\cite{AM-LecsMCInfDimMfs}.
Their construction builds on
open sets being mapped to open sets under a forward
flow. But this is not true for $\varphi_s$---from a
topological point of view the heat semi-flow is useless!
The way out was the simple observation
that preimages of open sets are open
by continuity of $\varphi_s$.
Unfortunately, still
the Abbondandolo-Majer method would
not apply, because things were moving in the
wrong direction now. However, the definition given
in~\cite[proof of Lemma~3.2]{Dietmar-BLMS}
in finite dimensions
carries over providing a Conley pair $(N,L)$ for the
semi-flow invariant set given by the critical point $x$.
Now the backward $\lambda$-Lemma enters. 
In~\cite{Joa-JOAOPESSOA,Joa-CONLEY} we use it
to define an invariant stable foliation of $(N,L)$
which is a fundamental ingredient in our construction
of a Morse filtration of $\Lambda M$
by open semi-flow invariant sets. In
Section~\ref{sec:first-application} we
discuss the key calculation.
\\
In other words, we were led to discover
the backward $\lambda$-Lemma through
the attempt to solve a very different problem---thereby
reconfirming a major principle advocated
by V.I.~Arnol'd throughout his mathematical life.
\begin{figure}
  \includegraphics{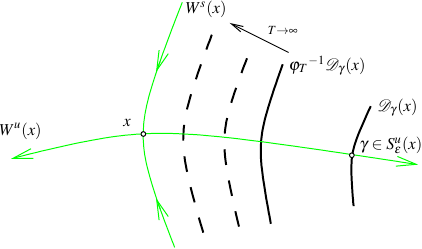}
  \caption{Preimage of disk $\Dd_\gamma(x)$ converges
           in $C^1$ and locally near
           $x$ to the stable manifold $W^s(x)$
           }
  \label{fig:fig-stable-fibr-global}
\end{figure}

\subsection{Main results}
Assume $\Dd_\gamma(x)$ is a disk in the Hilbert
manifold $\Lambda M=W^{1,2}(S^1,M)$
which intersects the
unstable manifold $W^u(x)$ transversally in a point
$\gamma$ near $x$. Our main
goal is to prove that the preimage
${\varphi_T}^{-1}\Dd_\gamma(x)$ converges, as
$T\to\infty$, uniformly in $C^1$ and locally near $x$ to
the stable manifold $W^s(x)$; see
Figure~\ref{fig:fig-stable-fibr-global}. In fact, we
prove right away a family version where $\Dd$ is fibered
over a descending sphere
$S^u_\eps(x)=W^u(x)\cap\{\Ss_V=\Ss_V(x)-\eps\}$.

Since the $\lambda$-Lemma is a local result
we choose a local parametrization
\begin{equation*}
      \Phi
      :=\exp_x:
      X 
      \to\Lambda M
      ,\qquad
      X=T_x \Lambda M=W^{1,2}(S^1,x^*TM),
\end{equation*}
of an open neighborhood of $x$ in $\Lambda M$
in terms of the exponential map; here compactness of
$M$ enters. The orthogonal splitting
\begin{equation*}
     X
     \simeq T_x W^u(x)\oplus T_x W^s(x)
     =X^-\oplus X^+
\end{equation*}
with associated orthogonal projections $\pi_\pm$
is a key ingredient to make the analysis work;
at this stage take the final identity as a definition.
By a standard graph argument we
may assume without loss
of generality that $\Uu$ is of the form $W^u\times \Oo^+$.
Here $W^u\subset X^-$ represents
a descending disk
$W^u(x)\cap\{\Ss_\Vv>\Ss_\Vv(x)-\delta\}$
for some $\delta>0$ sufficiently small
and $\Oo^+\subset X^+$ is an open ball about $0$.
By $\phi$ we denote the local semi-flow on $\Uu$
which represents the heat flow $\varphi$ on $\Lambda M$
with respect to the local
parametrization $\Phi$;
see~(\ref{eq:cauchy-local}).

\begin{hypothesis}[Local setup---Figure~\ref{fig:fig-local-setup}]\label{hyp:loc-coord-lambda}
Fix a perturbation $V\in C^\infty(S^1\times M)$
and a nondegenerate critical point $x$ of $\Ss_V$
of Morse index $k$.
\\
(a)~Consider the coordinates on $\Lambda M$
provided by $\Phi$ and modelled on the open
subset $\Uu=W^u\times \Oo^+$ of $X$.
In these coordinates the origin $0\in X$ represents
$x$ and $\Ss:=\Ss_V\circ\Phi^{-1}$ represents the
action. We denote closed radius $R$ balls about $0$ by
$$
     \Bb_R
     :=\{\Norm{\cdot}_X\le R\}
     ,\qquad
     \Bb^+_R
     :=\{\Norm{\cdot}_{X^+}\le R\}.
$$
Choose the constant $\rho_0>0$ in the Lipschitz
Lemma~\ref{le:f} smaller, if necessary, such that
$\Bb_{\rho_0}\subset\Uu$.
Pick a sufficiently small constant $\eps_0>0$
such that for each $\eps\in(0,\eps_0]$ the descending
and ascending disks
$$
     W^u_\eps(x)
     :=W^u(x)\cap\{\Ss_V>\Ss_V(x)-\eps\}
     ,\quad
     W^s_\eps(x)
     :=W^s(x)\cap\{\Ss_V<\Ss_V(x)+\eps\},
$$
are contained in the coordinate patch
$\Phi(\Bb_{\rho_0})$ and such that their closures are
diffeomorphic to the closed unit disks in $\R^k$
and $X^+$, respectively.
Existence of $\eps_0$ follows by the Morse-
and the Palais-Morse lemma.
\\
(b)~Fix $\mu\in(0,d)$ in the spectral gap~(\ref{eq:mu})
of the Jacobi operator. Pick $\varkappa\in(0,\rho_0)$ so small
that $\Dd:=S^u_\eps\times \Bb^+_\varkappa$ is
contained in $\Bb_{\rho_0}$ and set
$\Dd_\gamma:=\{\gamma\}\times\Bb^+_\varkappa$.
\\
(c)~Our notation for objects expressed in coordinates
will be the global notation with $x$ omitted,
for example $W^s_\eps$ and $\Dd_\gamma$.
\end{hypothesis}

\begin{theorem}[Backward $\lambda$-Lemma]
\label{thm:backward-lambda-Lemma}
Assume the local setup of
Hypothesis~\ref{hyp:loc-coord-lambda}.
In particular, consider the hyperbolic fixed point
$0$ of the local semi-flow $\phi$
defined by~(\ref{eq:cauchy-local})  on $\Uu\subset X$
and the hypersurface 
$\Dd=S^u_\eps\times\Bb^+_\varkappa\subset\Bb_{\rho_0}\subset\Uu$;
see Figure~\ref{fig:fig-lambda-Lemma}.
Then the following is true.
There is a closed ball $\Bb^+\subset X^+$
of radius $r>0$ about zero, a constant
$T_0>0$, and a Lipschitz continuous map
\begin{equation*}
\begin{split}
     \Gg:(T_0,\infty)\times S^u_\eps\times\Bb^+
    &\to W^u\times\Bb^+\subset\Uu
   \\
     (T,\gamma,z_+)
    &\mapsto 
     \left(G^T_\gamma(z_+),z_+\right)
     =:\Gg^T_\gamma(z_+)
\end{split}
\end{equation*}
of class $C^1$ and defined by~(\ref{eq:G^T}).
Each map $\Gg^T_\gamma:\Bb^+\to X$
is bi-Lipschitz, a diffeomorphism onto its image, and
$\Gg^T_\gamma(0)=\phi_{-T}\gamma=:\gamma_T$.
The graph of $G^T_\gamma$ consists of those $z\in \Uu$
which satisfy $\pi_+ z\in \Bb^+$ and reach the fiber 
$\Dd_\gamma=\{\gamma\}\times\Bb^+_\varkappa$
at time $T$, that is
$$
     \Gg^T_\gamma(\Bb^+)
     ={\phi_T}^{-1}
     \Dd_\gamma\cap
     \left( X^-\times\Bb^+\right).
$$
Furthermore, the graph map $\Gg^T_\gamma$ converges
uniformly, as $T\to\infty$, to the stable manifold
graph map $\Gg^\infty$ of 
Theorem~\ref{thm:local-stable-manifold-graph}.
More precisely, it holds that~\footnote{
  Note that the difference lies in $X^-$, hence in $C^\infty$.
  Therefore it makes sense to take the $C^1$
  norm.
  }
\begin{equation*}
\begin{split}
     \Norm{\Gg^T_\gamma(z_+)-\Gg^\infty(z_+)}_{C^1(S^1)}
     \le \rho_0 e^{-T\frac{\mu}{16}}
\end{split}
\end{equation*}
for all $T>T_0$, $\gamma\in S^u_\eps$, and $z_+\in\Bb^+$.
\end{theorem}

\begin{figure}
  \includegraphics{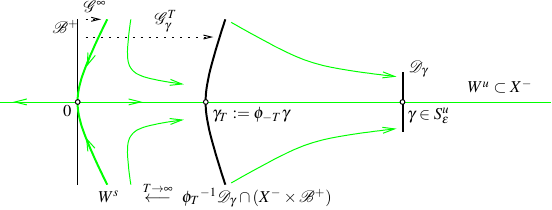}
  \caption{Backward $\lambda$-Lemma}
  \label{fig:fig-lambda-Lemma}
\end{figure}

\begin{theorem}
[Uniform $C^1$ convergence on $X$]
\label{thm:uniform-C1}
Under the assumptions of
Theorem~\ref{thm:backward-lambda-Lemma}
the linearized graph maps
$d\Gg^T_\gamma(z_+):X^+\to X$
extend to bounded linear operators
on the $L^2$ completions
and their limit, as $T\to\infty$, is
$d\Gg^\infty(z_+)$, uniformly in $z_+\in\Bb^+$.
More precisely,
$$
     \Norm{d\Gg^T_\gamma(z_+)
     v}_2
     \le 2 \Norm{v}_2
$$
and
$$
     \Norm{d\Gg^T_\gamma(z_+)v-d\Gg^\infty(z_+)v}_2
     \le e^{-T\frac{\mu}{16}}\Norm{v}_2
$$
for all $T>T_0$,  $\gamma\in S^u_\eps$,
$z_+\in\Bb^+$, and $v$ in the $L^2$ closure of $X^+$.
\end{theorem}

\begin{remark}\label{rem:general-p}
(i)~The reason why the $W^{1,2}$ norm
has been replaced by $C^1$ ($\hookleftarrow W^{1,4}$) in
Theorem~\ref{thm:backward-lambda-Lemma} and by
$L^2$ in Theorem~\ref{thm:uniform-C1}
is the application in Section~\ref{sec:first-application}.
Here the $L^2$ nature of~(\ref{eq:heat}) requires
to estimate the nonlinearity $f$
in~(\ref{eq:cauchy-local}) in the $L^2$ norm.
But $f$ maps $W^{1,4}$ to $L^2$.

(ii)~All results in this paper extend to the more general
class of perturbations satisfying axioms~(V0-V3)
in~\cite{SaJoa-LOOP}; see~\cite{Joa-LOOPSPACE}.
\end{remark}

\begin{remark}
(a)~Theorem~\ref{thm:backward-lambda-Lemma}
recovers the common case of a single disk intersecting
the unstable manifold transversely in one point
$\gamma$ near $0$.
Apply the implicit function theorem to bring
the disk into the normal form $\{\gamma\}\times\Bb^+_\varkappa$.
Observe that $\Gg^T_\gamma$ is defined without
reference to any neighbors of $\gamma$.
To formally get a bundle
$\Dd=S^u_\eps\times\Bb^+_\varkappa$
as in the hypothesis just add disks artificially.
\\
(b)~Theorem~\ref{thm:backward-lambda-Lemma}
for endpoint time $T=\infty$ and radius $\varkappa=0$
recovers two known results.
These appear as extreme cases concerning the radius
$0$ disk bundle $\Dd=S^u_\eps\times\{0\}$.
\begin{enumerate}
\item[I]
{\it Disk of radius $0$ sitting at the origin}:
  In this case $S^u_\eps=\{0\}$, that is
  the disk bundle degenerates to
  just one radius $0$ disk sitting at the origin.
  This recovers the local stable manifold
  Theorem~\ref{thm:local-stable-manifold-graph}
  and inspires the notation $\Gg^\infty$ for the
  stable manifold graph map.
  The preimage ${\phi_T}^{-1}(0)$ for $T=\infty$
  corresponds to the local stable manifold.
\item[II]
{\it Radius $0$ disk bundle sitting at $\infty$}:
  This recovers the stable foliation
  in~\cite{CHOW-LIN-LU}.
  Two points belong to the
  same leaf if under the semi-flow $\phi_s$
  their difference converges exponentially
  to zero, as $s\to\infty$.
  The leaf over $0$ is the local stable manifold.
\end{enumerate}
\end{remark}

\begin{remark}[Mixed Cauchy problem]
\label{rem:mixed-Cauchy}
We motivate why the map $\Gg$ in the backward
$\lambda$-Lemma should exist.
Assume the hypotheses of the backward
$\lambda$-Lemma and fix $T>0$. Each point 
$z$ in the preimage ${\phi_T}^{-1}\Dd_\gamma$
corresponds to a unique semi-flow line $\xi$ such that
$\xi(0)=z$ and $\xi$ hits the fiber
$\Dd_\gamma$ precisely at time $T$, say in the point
$q:=\xi(T)$. Of course, we cannot change the order,
i.e. first choosing an end point $q\in\Dd_\gamma$ and
then determining a semi-flow line $\xi$ with $\xi(T)=q$.
This would amount to solve the Cauchy
problem for the heat equation in \emph{backward} time,
a problem well known to be ill defined in general:
Indeed any non-smooth element
$q\in\Dd_\gamma\subset W^{1,2}$ cannot be reached, since
the point $\xi(T)$ on any heat flow trajectory
$\xi$ is necessarily a smooth loop in $M$---due to the strongly
regularizing effect of the heat flow for $T>0$; see
e.g.~\cite{Joa-LOOPSPACE}.
However, consider the splitting $X=X^-\oplus X^+$
in unstable and stable tangent spaces.
In Section~\ref{sec:splittings} we will see that
each element of $X^-$ is smooth. So specifying
only the $X^-$ part of the endpoint does
not contradict regularity to start with.
The key idea is to introduce the notion
of a {\bf mixed Cauchy problem}: Apart from time
$T$ only the stable part $z_+$ of the initial point
is prescribed---in exchange of prescribing
in addition the unstable part $q_u$
of the end point; see Figure~\ref{fig:fig-mixed-Cauchy}.
Indeed the representation formula in
Proposition~\ref{prop:representation-formula}
shows that the mixed Cauchy problem is equivalent to the
usual Cauchy problem with initial value $z$.
Since the latter admits a unique solution,
so does the mixed Cauchy problem.
\end{remark}

\begin{figure}
  \includegraphics{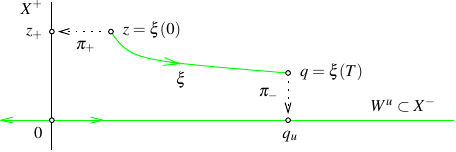}
  \caption{Data $(T,q_u,z_+)$
           for mixed Cauchy problem
           }
  \label{fig:fig-mixed-Cauchy}
\end{figure}

\subsection{Outlook}
To put the backward $\lambda$-Lemma and the
associated stable foliations~\cite{Joa-CONLEY} in perspective
recall the celebrated proof by Palis
in his 1967 PhD thesis~\cite{Palis-PhD}
of Andronov-Pontryagin
structural stability of hyperbolic dynamical systems
in small dimensions.
Key innovations and tools in his proof
were the notions of stable and unstable foliations
which led to numerous applications ever since.

Another interesting perspective of the backward
$\lambda$-Lemma is that, to the best of our
knowledge, it provides the first
\emph{backward} time information
concerning the heat flow---apart from the obvious
backward flow on the (finite dimensional) unstable manifolds.
In fact the backward $\lambda$-Lemma provides
backward time information on \emph{open} subsets.

Going from~(\ref{eq:heat}) and $W^{1,2}$ to general semilinear
parabolic PDEs and $W^{k,p}$ will be investigated elsewhere.

\subsection{Application}
\label{sec:first-application}
The method in~\cite{Joa-JOAOPESSOA,Joa-CONLEY} to construct
a Morse filtration of the loop space
is inspired
by Conley theory~\cite{CONLEY-CBMS}.
For reals $\eps,\tau>0$ denote by
$N$ the path connected component
of $x$ of the open set
$
     \{\gamma\in\Lambda M\mid\text{$\Ss_V(\gamma)
      <\Ss_V(x)+\eps$,
     $\Ss_V(\varphi_\tau \gamma)>\Ss_V(x)-\eps$}\}
$.
For $\eps>0$ small and $\tau>0$ large
the closed subset
$L:=\{\gamma\in N\mid\text{$\Ss_V(\varphi_{2\tau}\gamma)
\le \Ss_V(x)-\eps$}\}$ of $N$ is an {\bf exit set} of $N$
and $(N,L)$ is a {\bf Conley pair} for the
semi-flow invariant set $\{x\}$;
see~\cite{Joa-CONLEY,Joa-LOOPSPACE}.
Concerning the Morse filtration a 
fundamental step is to prove that
relative singular homology is given by
\begin{equation*}
     \Ho_\ell(N,L)
     \simeq
     \begin{cases}
        \Z& ,\ell=k:=\IND(x),
        \\
        0&  \text{, otherwise.}
     \end{cases}
\end{equation*}
Because the part of $N$ in the unstable manifold
is an open $k$ disk bounded by the
(relatively) closed annulus $L\cap W^u(x)$,
the relative homology of these parts
has the required property
and it suffices to show that
$(N\cap W^u(x),L\cap W^u(x))$ is a
deformation retract of $(N,L)$.
Observe that $N$ contains the ascending disk
$W^s_\eps(x)$---from now on abbreviated
$W^s_\eps$---whose part in the unstable manifold
is precisely the critical point $x$ itself.
Thus on $W^s_\eps$ the semi-flow $\varphi_s$ itself
provides the desired deformation.
Obviously this fails on the complement of
the ascending disk.
Now the backward $\lambda$-Lemma comes in.
It naturally endows $(N,L)$, as we show
in~\cite{Joa-CONLEY}, with the structure of a 
codimension $k$ foliation
whose leaves $N(\gamma_u)$ are
parametrized by $\gamma_u\in N\cap W^u(x)$.
Furthermore, each leaf is
diffeomorphic to a neighborhood $U_{\gamma_u}$
of $W^s_\eps$ in $W^s(x)$. The leaf through $x$
is given by $N(x)=W^s_\eps=U_x(W^s_\eps)$.
These diffeomorphisms, denoted by
$$
    \Psi_{\gamma_u}:U_{\gamma_u}(W^s_\eps)
     \stackrel{\simeq\,\,}{\longrightarrow}
     N(\gamma_u)
     ,\qquad
     \Psi_{\gamma_u}(x)=\gamma_u
     ,\qquad
     \Psi_0=id_{W^s_\eps},
$$
allow to extend to all of $N$ the desirable
deformation property provided by $\varphi_s$
on $W^s_\eps$.
Indeed pick $\gamma\in N$. By the foliation
property $\gamma$ lies on some leaf,
say on $N(\gamma_u)$. Now the map
$
     \theta_s(\gamma)
     :=\Psi_{\gamma_u}\circ\varphi_s
     \circ{\Psi_{\gamma_u}}^{-1}(\gamma)
$
for
$
     s\in[0,\infty]
$
deforms $N$ onto its part in the unstable manifold.
So we are done.
Well, note the subtlety arising due to
the deformation having to take place
entirely in $N$ which is equivalent to invariance
of $U_{\gamma_u}(W^s_\eps)$ under $\varphi_s$.
For $U_x(W^s_\eps)=W^s_\eps$ this follows 
immediately from the
fact that the action decreases along the heat flow.
Since $\dim U_{\gamma_u}(W^s_\eps)=\infty$,
the general case is non-trivial.
Apart from the Palais-Smale condition,
the analytic properties of the graph maps
$\Gg^T_\gamma$ provided by
Theorems~\ref {thm:backward-lambda-Lemma}
and~\ref{thm:uniform-C1} enter heavily.
We refer to~\cite{Joa-CONLEY} for details
and to~\cite{Joa-JOAOPESSOA} for a survey.

\section{Toolbox}\label{sec:toolbox}

Throughout we fix a nondegenerate critical point $x$
of $\Ss_V:\Lambda M\to\R$.
Representing the Hessian of
$\Ss_V$ at $x$ with respect to the $L^2$ inner product
on the loop space gives rise to the
{\bf\boldmath Jacobi operator $A_x$} defined by
\begin{equation}\label{eq:Jacobi-1}
     A_x\xi
     =-\Nabla{t}\Nabla{t}\xi-R(\xi,\dot x)\dot x
     -\Nabla{\xi}\nabla V_t(x)
\end{equation}
for every smooth vector field $\xi$ along the loop $x$.
Here $R$ denotes the Riemannian curvature tensor.
Viewed as unbounded operator
on a general Sobolev space
$
     W^{k,q}:=W^{k,q}(S^1,x^*TM)
$
with dense domain $W^{k+2,q}$, where $k\in\N_0$
and $q\ge 1$, the spectrum of $A_x$ does not
depend on $(k,q)$ and takes the
form of a sequence
of real eigenvalues (counting multiplicities)
\begin{equation}\label{eq:spec-A}
     \lambda_1\le\lambda_2\le\ldots\le\lambda_k<0<
     \lambda_{k+1}\le\lambda_{k+2}\le\ldots
\end{equation}
which converges to $\infty$.
Calculation of the spectrum is standard:
One picks the Hilbert space
case $(k,q)=(0,2)$ and proves first that $A_x$ admits
a compact self-adjoint resolvent. In the second step
it remains to prove $C^\infty$ regularity of
eigenfunctions.
The {\bf spectral gap} $(0,d)$ of $A_x$
is determined by
\begin{equation}\label{eq:mu}
     d:=dist\left(0,\sigma(A)\right)
     =\min\{-\lambda_k,\lambda_{k+1}\}>0.
\end{equation}
By $\sigma_\pm$ we denote the {\bf positive} and
{\bf\boldmath negative part of the spectrum of $A_x$}.
Note that nondegeneracy of the critical point $x$
means that zero is not in the spectrum of $A_x$.
Equivalently $x$ is a {\bf hyperbolic} fixed
point of $\varphi_s$ whenever $s>0$, that is the
spectrum of the linearized flow $d\varphi_s(x)$ does
not contain $1$.
The {\bf Morse index} of $x$ is the number $k$
of negative eigenvalues of $A_x$ counted with
multiplicities.

It is worthwhile to mention some of the useful
properties enjoyed by the action functional:
It strictly decreases along non-constant
heat flow trajectories. It is bounded below
and satisfies the Palais-Smale condition.

In the following subsections we provide
the analytical tools required in the proof 
of the backward $\lambda$-Lemma.
Apart from Lemma~\ref{le:f} they are all
well known, surely by the experts,
and so we simply list
them without proofs. On the other hand,
some are difficult to find in the literature,
e.g. sectoriality of $A_x$ in the relevant
periodic case. So here is some
good news for non-experts:

\noindent
{\bf Convention.} {The proof of any assertion
attributed {\bf well known} in
Section~\ref{sec:toolbox}
is given in~\cite{Joa-LOOPSPACE}.
The same holds for facts stated
{\bf without reference}.}

\subsection{Local semi-flow}

Recall that $\Uu=W^u\times \Oo^+$
by Hypothesis~\ref{hyp:loc-coord-lambda}.
Any path $[0,T]\ni s\mapsto u_s$ in the
neighborhood $\Uu(x):=\Phi(\Uu)$ of $x$ in
the Hilbert manifold $\Lambda M$ 
corresponds to a path $\zeta:[0,T]\to \Uu$,
$s\mapsto\zeta(s)$, determined uniquely
by the identity
$
     u_s=\exp_{x}\zeta(s)=:\Phi(\zeta(s))
$
pointwise for $t\in S^1$.
Applying the operators $\p_s$ and $\Nabla{t}\Nabla{t}$
to this identity transforms the Cauchy
problem on $\Lambda M$ associated
to~(\ref{eq:heat}) into the
equivalent Cauchy problem
\begin{equation}\label{eq:cauchy-local}
     \frac{d}{ds}\zeta(s)+A_x\zeta(s)
     =f(\zeta(s)),\qquad
     \zeta(0)=z:=\Phi^{-1}(\gamma)\in\Uu,
\end{equation}
for maps $\zeta:[0,T]\to\Uu\subset X=W^{1,2}(S^1,x^*TM)$
where $T$ depends on $z\in\Uu$ and $A_x$ denotes
the Jacobi operator~(\ref{eq:Jacobi-1}) on $W^{1,2}$
with dense domain $W^{3,2}$.
The solution to~(\ref{eq:cauchy-local})
defines the {\bf local semi-flow} $\phi_sz:=\zeta(s)$
on $\Uu$ that represents the heat flow.
The nonlinearity
$$
     f:X\supset\Uu\to Y
     ,\qquad X=W^{1,2}
     ,\qquad Y=L^1,
$$
actually maps $W^{1,2p}$ to $L^p$
for $p\ge1$ and is given by the identity
\begin{equation}\label{eq:f}
\begin{split}
     f(\zeta)
    &=
     E_2(x,\zeta)^{-1}
     \Bigl(
     E_{11}(x,\zeta)\bigl(\dot x,\dot x\bigr)
     +2E_{12}(x,\zeta)
     \bigl(\dot x,\Nabla{t}\zeta\bigr)
     \\
    &\quad
     +E_{22}(x,\zeta)
     \bigl(\Nabla{t}\zeta,
     \Nabla{t}\zeta\bigr)
     +\nabla V_t(\exp_x\zeta)
     -E_1(x,\zeta)\,\nabla V_t(x)
     \Bigr)
     \\
    &\quad
     -R(\zeta,\dot x)\dot x-\Nabla{\zeta}\nabla V_t(x)
\end{split}
\end{equation}
pointwise at $(s,t)$.
To arrive at this form of $f$ we used the well known
covariant partial derivatives
of the exponential map $E(q,v):=\exp_q v$.
These are multilinear maps
$$
     E_{i_1\dots i_j}(q,v):\left(T_qM\right)^{\times j}
     \to T_{\exp_q v}M
$$
which depend smoothly on $(q,v)\in TM$ for each $j\in\N$.
Those up to order two are characterized by the identities
\begin{equation*}
\begin{split}
     \tfrac{d}{dt}\exp_\gamma(\xi)
    &=E_1(\gamma,\xi)\p_t\gamma
     +E_2(\gamma,\xi)\Nabla{t}\xi
    \\
     \Nabla{t}\left( E_1(\gamma,\xi)\eta\right)
    &=E_{11}(\gamma,\xi)\left(\eta,\p_t \gamma\right)
      +E_{12}(\gamma,\xi)\left(\eta,\Nabla{t}\xi\right)
      +E_1(\gamma,\xi)\Nabla{t}\eta
    \\
     \Nabla{t}\left( E_2(\gamma,\xi)\eta\right)
    &=E_{21}(\gamma,\xi)\left(\eta,\p_t \gamma\right)
      +E_{22}(\gamma,\xi)\left(\eta,\Nabla{t}\xi\right)
      +E_2(\gamma,\xi)\Nabla{t}\eta
\end{split}
\end{equation*}
whenever $\gamma:\R\to M$, $t\mapsto\gamma(t)$,
is a smooth curve in $M$ and $\xi,\eta$ are smooth
vector fields along $\gamma$.
These covariant derivatives satisfy
\begin{equation}\label{eq:E_ij(0)}
     E_1(q,0)=E_2(q,0)=\1
     ,\quad
     E_{11}(q,0)=E_{21}(q,0)=E_{22}(q,0)=0,
\end{equation}
and admit symmetries
\begin{equation*}
     E_{12}(q,z)\left(v,w\right)
     =E_{21}(q,z)\left(w,v\right)
     ,\quad
     E_{22}(q,z)\left(v,w\right)
     =E_{22}(q,z)\left(w,v\right)
     ,
\end{equation*}
and
\begin{equation*}
     E_{11}(q,z)\left(v,w\right)
     -E_{11}(q,z)\left(w,v\right)
     =E_{2}(q,z) R(v,w)z
\end{equation*}
for all $q\in M$ and $z,v,w\in T_qM$.
Furthermore, it holds that
\begin{equation}\label{eq:E112=R}
     E_{112}(\gamma,0)
     \left(\dot\gamma,\dot\gamma,\xi\right)
     :={\textstyle \left.\frac{D}{d\tau}\right|}_{\tau=0}
     E_{11}(\gamma,\tau\xi)
     \left(\dot\gamma,\dot\gamma\right)
     =R(\xi,\dot\gamma)\dot\gamma
\end{equation}
pointwise for every $t\in\R$. For more details
see e.g.~\cite[Appendix]{Joa-LOOPSPACE}.

After all what is the advantage of
reformulating the Cauchy problem?
Obviously the linear structure of $X$ to start with.
However, the really great features
are a)~ the spectral splitting
$X\simeq X^-\oplus X^+$ induced by $A_x$
is preserved by the semigroups of
Section~\ref{sec:splittings} and b)~the part
$X^-$ is of finite dimension $k$
and consists of smooth elements.

\subsection{Lipschitz estimate for the nonlinearity}

\begin{lemma}[Locally Lipschitz]\label{le:f}
   There are constants $\rho_0,\kappa_*>0$
   and a continuous nondecreasing function $\kappa$
   on the interval $[0,\rho_0]$ with $\kappa(0)=0$
   such that the following is true for any constant
   $p\ge 1$. In the Sobolev space
   $W^{1,2p}(S^1,x^*TM)$ consider the
   closed ball $\Bb^{1,2p}_{\rho_0}$ of radius $\rho_0$.
   Then $\Bb^{1,2p}_{\rho_0}\subset\Uu$ and the nonlinearity
   $f:\Bb^{1,2p}_{\rho_0}\to L^p$ given
   by~(\ref{eq:f}) is of class $C^1$ and satisfies
   $f(0)=0$, $df(0)=0$, and
   \begin{equation*}
   \begin{split}
     \Norm{f(\xi)-f(\eta)}_p
    &\le \kappa(\rho)
     \Norm{\xi-\eta}_{1,2p},
    \\
     \Norm{df(\xi)v-df(\eta)v}_p
    &\le \kappa_*
     \Norm{\xi-\eta}_{1,2p}
     \Norm{v}_{1,2p},
   \end{split}
   \end{equation*}
   whenever $\norm{\xi}_{1,2p},\norm{\eta}_{1,2p}
   \le\rho<\rho_0$ and $v\in W^{1,2p}$.
\end{lemma}

\begin{corollary}\label{cor:f}
Assume Lemma~\ref{le:f}. Then
$
     \norm{df(\xi)v}_p
     \le \kappa(\rho)
     \norm{v}_{1,2p}
$
whenever $v\in W^{1,2p}$ and
$\norm{\xi}_{1,2p}\le\rho<\rho_0$.
\end{corollary}

\begin{proof}
Use that 
$df(\xi)v
=\lim_{\tau\to 0}\frac{f(\xi+\tau v)-f(\xi)}{\tau}$
and apply Lemma~\ref{le:f}.
\qed
\end{proof}

\begin{remark}
(a)~That $\kappa(\rho)$ is 
nondecreasing in $\rho$ is used to prove the assertion
of Theorem~\ref{thm:local-stable-manifold-graph}
that at the fixed point $0$
the stable manifold is tangent to $X^+$.
\\
(b)~The Lipschitz estimate 
for $df$ with constant $\kappa_*$ is required,
firstly, to prove that the graph map $\Gg^T_\gamma$
is of class $C^1$ in $T$ and, secondly, to prove
uniform convergence of its derivative
to the derivative 
of the stable manifold graph map, as 
$T\to\infty$; see proofs of
Theorem~\ref{thm:backward-lambda-Lemma} step~4
and Theorem~\ref{thm:uniform-C1} step~II.
\\
(c)~While Lemma~\ref{le:f} is used mainly in case $p=1$
the fundamental step in Section~\ref{sec:first-application},
carried out in~\cite{Joa-CONLEY},
requires Lemma~\ref{le:f} for $p=2$.
\end{remark}

\begin{proof}[Lipschitz Lemma~\ref{le:f}]
Fix $p\ge1$ and observe that
$\norm{\cdot}_\infty\le c_\infty\norm{\cdot}_{1,2}
\le c_\infty\norm{\cdot}_{1,2p}$ for some constant $c_\infty>0$.
The last step uses H\"older's inequality
and $\mathrm{vol}(S^1)=1$.
By $\iota>0$ we denote the injectivity radius
of the closed Riemannian manifold $M$.
Fix $\rho_0\in(0,\iota/8c_\infty)$ sufficiently small
such that the ball $\Bb_{\rho_0}:=\Bb_{\rho_0}^{1,2}$
in $W^{1,2}$ of radius $\rho_0$
is contained in $\Uu$.
From now on assume that $\rho\in[0,\rho_0)$ and
$\xi,\eta\in W^{1,2p}$ satisfy
$\norm{\xi}_{1,2p},\norm{\eta}_{1,2p}\le\rho$.
Note that $\xi\in\Uu$ and
$\norm{\xi}_\infty\le c_\infty\rho<\frac{\iota}{8}$
and similarly for $\eta$.
So both $\xi$ and $\eta$ take values in
the compact subset $D_\rho\subset TM$
consisting of all pairs $(q,v)$
such that $q\in M$ and $v\in T_qM$ satisfies
$\abs{v}\le 4c_\infty \rho$.
Note that ${D}_0$ is the zero section.

To see that $f(0)=0$ use~(\ref{eq:f})
and~(\ref{eq:E_ij(0)}).
Use in addition~(\ref{eq:E112=R}) to prove that
$df(0)\zeta:=\frac{d}{d\tau}f(\tau\zeta)=0$.
Abbreviate
$
     X:=\eta-\xi
$
to obtain that
\begin{equation*}
\begin{split}
     f(\xi)-f(\eta)
    &=\left(
     E_2(x,\xi)^{-1}
     E_{11}(x,\xi)
     -E_2(x,\eta)^{-1}
     E_{11}(x,\eta)\right)
     \bigl(\dot x,\dot x\bigr)
     +R(X,\dot x)\dot x\\
    &
     +2\left(
     E_2(x,\xi)^{-1}
     E_{21}(x,\xi)\Nabla{t}\xi 
     -E_2(x,\eta)^{-1}
     E_{21}(x,\eta)\Nabla{t}
     \eta\right)
     \dot x\\
    &
     +E_2(x,\xi)^{-1}
     E_{22}(x,\xi)
     \bigl(\Nabla{t}\xi,\Nabla{t}
     \xi\bigr)
     -E_2(x,\eta)^{-1}
     E_{22}(x,\eta)
     \bigl(\Nabla{t}\eta,
     \Nabla{t}\eta\bigr)\\
    &
     +E_2(x,\xi)^{-1}
     \nabla V_t(\exp_{x}\xi)
     -E_2(x,\eta)^{-1}
     \nabla V_t(\exp_{x}\eta)
     +\Nabla{X}\nabla V_t(x)\\
    &
     -\left(E_2(x,\xi)^{-1}
     E_1(x,\xi)
     -E_2(x,\eta)^{-1}
     E_1(x,\eta)\right)
     \nabla V_t(x)
\end{split}
\end{equation*}
pointwise at every $t\in S^1$.
We denote the last five lines
of the formula above by $I$ 
through $V$, respectively,
and deal with each one separately.
For now think of $\xi$
as a fixed parameter and view
$
     \eta(X)=\xi+X
$
as a function of $X$.
Then each line becomes a (smooth) function of $X(t)$
depending on additional quantities such as certain
derivatives of $\xi$, $X$, and $x$ all evaluated at $t$.
For instance, term $I$ becomes the identity
\begin{equation*}
     I(X)
     =\left(
     E_2(x,\xi)^{-1}E_{11}(x,\xi)
     -E_2(x,\eta(X))^{-1}E_{11}(x,\eta(X))
     \right)
     \bigl(\dot x,\dot x\bigr)
     +R(X,\dot x)\dot x
\end{equation*}
pointwise at every $t\in S^1$.
Straighforward calculation shows that
\begin{equation*}
\begin{split}
    &dI(X) Y
     =\left.\tfrac{D}{d\tau}\right|_{\tau=0}
     I(X+\tau  Y)\\
    &=\bigl(E_2(x,\eta(X))\, \bigr)^{-1}
     E_{22}(x,\eta(X))
      \Bigl(E_2(x,\eta(X))^{-1}
      E_{11}(x,\eta(X))
      \left(\dot x,\dot x\right),
      Y\Bigr)\\
    &\quad
     -E_2(x,\eta(X))^{-1}
     E_{112}(x,\eta(X))
     \left(\dot x,\dot x,Y\right)
     +R(Y,\dot x)\dot x
\end{split}
\end{equation*}
pointwise at every $t\in S^1$.
Note that $\eta(\sigma X)=\sigma\eta+(1-\sigma)\xi$,
for $\sigma\in[0,1]$ and pointwise in $t$,
takes values in
${D}_{\rho/2}\subset{D}_\rho\subset{D}_{\rho_0}$.
Note further that $I(0)=0$.
Hence by Taylor's theorem there is
a constant $\sigma=\sigma(t)\in[0,1]$ such that
\begin{equation*}
\begin{split}
     \Abs{I(X)}
     =\Abs{dI(\sigma X) X}
    &\le
     \Norm{{E_2}^{-1}}_{L^\infty({D}_{\rho_0})}^2
     \Norm{E_{22}}_{L^\infty({D}_\rho)}
     \Norm{E_{11}}_{L^\infty({D}_\rho)}
     \Abs{\dot x}^2\Abs{X}
     \\
    &\quad 
     +\Norm{{E_2}^{-1} 
     E_{112}\left(*,*,\cdot\right)
     -R(\cdot,*) *}_{L^\infty({D}_\rho)}
     \Abs{\dot x}^2\Abs{X}
     \\
    &=:\kappa_1(\rho)\Abs{\dot x}^2\Abs{X}
\end{split}
\end{equation*}
pointwise at every $t\in S^1$. 
The function $\kappa_1$
depends continuously on $\rho\in[0,\rho_0]$
and that $\kappa_1(0)=0$ by
the curvature identity~(\ref{eq:E112=R})
and since $E_{ij}(\cdot ,0)=0$ for
$i,j\in\{1,2\}$ by~(\ref{eq:E_ij(0)}). By the a priori
estimate~\cite[Theorem~12]{Joa-HEATMORSE-II}
applied to the constant heat flow trajectory
$u(s)\equiv x$ there is a constant
$C=C(V,\Ss_\Vv(x))$ such that
$\abs{\dot x(t)}\le\norm{\dot x}_\infty\le C$.
By H\"older's inequality
$\norm{gh}_p\le\norm{g}_{2p}\norm{h}_{2p}$ we obtain
the desired Lipschitz estimate for term one, namely
$
     \norm{I(X)}_p
     \le \kappa_1(\rho)\norm{\dot x}_{2p}^2
     \norm{X}_\infty
     \le c_\infty\kappa_1(\rho)\norm{\dot x}_\infty^2
     \norm{\xi-\eta}_{1,2p}
$.
Indeed the constant depends on $\rho$,
but not on $p$.
We did not pull out $\norm{\dot x}_\infty$ right
in the first step in order to illustrate how
H\"older's inequality serves
to deal with first order squares.
The argument for terms two through five is analogous;
see~\cite{Joa-LOOPSPACE} for details. Here first
order squares of the form $\abs{\Nabla{t} X}^2$ appear.

To see that $f$ is of class $C^1$
observe that
$
     df(\xi) X
     =-dI(0)X-\ldots-dV(0) X
$.
Careful inspection term by term then
shows that each of the five terms in this sum
depends continuously on $\xi$
with respect to the $W^{1,2p}$ topology.

It remains to prove the second Lipschitz estimate,
that is the one for the difference of derivatives 
$df(\xi)-df(\eta)$.
Unfortunately, the number of terms appearing
during the calculation is rather large.
Fortunately, we are only claiming existence
of a constant $\kappa_*$.
Straightforward calculation shows that
\begin{equation*}
\begin{split}
     df(\xi)v
    &=
     -R(v,\dot x)\dot x
     -E_2(x,\xi)^{-1}
     \Bigl[
     E_{22}(x,\xi)\left(
     E_2(x,\xi)^{-1}
     E_{11}(x,\xi)\left(\dot x,\dot x\right),
     v\right)
   \\
    &\quad
     -E_{22}(x,\xi)
     \bigl(E_2(x,\xi)^{-1}
     \bigl[
     2E_{12}(x,\xi)\left(\dot x,\Nabla{t}\xi\right)
     +E_{22}(x,\xi)\left(\Nabla{t}\xi,\Nabla{t}\xi\right)
     \bigr]
     ,v\bigr)
   \\
    &\quad
     -E_{22}(x,\xi)\left(
     E_2(x,\xi)^{-1}
     \nabla V_t(\exp_x \xi),
     v\right)
     +E_{22}(x,\xi)\left(\nabla V_t(x),v\right)
   \\
    &\quad
     +E_{112}(x,\xi)\left(\dot x,\dot x,v\right)
     +2E_{122}(x,\xi)\left(\dot x,\Nabla{t}\xi,v\right)
     +2E_{12}(x,\xi)\left(\dot x,\Nabla{t} v\right)
   \\
    &\quad
     +E_{222}(x,\xi)
     \left(\Nabla{t}\xi,\Nabla{t}\xi,v\right)
     +2E_{22}(x,\xi)\left(\Nabla{t}\xi,\Nabla{t} v\right)
   \\
    &\quad
     +{\textstyle\left.\frac{D}{d\tau}\right|_{\tau=0}}
     \nabla V_t(\exp_x (\xi+\tau v))
     -E_{12}(x,\xi)\left(\nabla V_t(x),v\right)
     \Bigr]
     -\Nabla{v}\nabla V_t(x).
\end{split}
\end{equation*}
Denote the fourteen terms in this sum by
$\sum_{j=1}^{14}H_j(\xi)v$. For $X:=\eta-\xi$ set
$2F_j(X)v:=H_j(\xi)v-H_j(\xi+X)v$.
For instance, consider $F_8$. We get that
\begin{equation*}
\begin{split}
    &dF_8(X)\left( v,Y\right)
  \\&=\left.\tfrac{D}{d\tau}\right|_{\tau=0}
     \bigl[
     E_2(x,\xi+X+\tau Y)^{-1}
     E_{122}(x,\xi+X+\tau Y)
     \left(\dot x,\Nabla{t}\left(\xi+X+\tau Y\right)
     ,v\right)\bigr]
  \\&=-{E_2}^{-1}E_{22}\bigl({E_2}^{-1}E_{122}
    \left(\dot x,\Nabla{t}\xi+\Nabla{t} X,v\right)
    ,Y\bigr)
  \\&\quad
     +{E_2}^{-1}E_{1222}
    \left(\dot x,\Nabla{t}\xi+\Nabla{t} X,v,Y\right)
     +{E_2}^{-1}E_{122}
    \left(\dot x,\Nabla{t} Y,v\right)
\end{split}
\end{equation*}
where the maps are evaluated at $(x,\xi+X)$.
Since $F_8(0)=0$ there is by Taylor's theorem,
pointwise at every $t\in S^1$,
a constant $\sigma=\sigma(t)\in[0,1]$ such that
\begin{equation*}
\begin{split}
     \Norm{2F_8(X)v}_p
    &=\Norm{2dF_8(\sigma X)\left( v,X\right)}_p
  \\&\le\Norm{{E_2}^{-1}}_{L^\infty_{\rho_0}}
    \Bigl(
    \Norm{{E_2}^{-1}}_{L^\infty_{\rho_0}}
    \Norm{E_{22}}_{L^\infty_{\rho}}
    \Norm{E_{122}}_{L^\infty_{\rho_0}}
    +\Norm{E_{1222}}_{L^\infty_{\rho_0}}
    \Bigr)
  \\&\quad
    \cdot\Norm{\dot x}_{2p}
    \Bigl(\sigma\Norm{\eta}_{1,2p}
    +(1-\sigma)\Norm{\xi}_{1,2p}\Bigr)
    \Norm{v}_\infty\Norm{\xi-\eta}_\infty
  \\&\quad
    +\Norm{{E_2}^{-1}}_{L^\infty_{\rho_0}}
    \Norm{E_{122}}_{L^\infty_{\rho_0}}
    \Norm{\dot x}_{2p}
    \Norm{\xi-\eta}_{1,2p}
    \Norm{v}_\infty
\end{split}
\end{equation*}
where we abbreviated
$L^\infty_{\rho}:=L^\infty({D}_{\rho})$.
Since $\Norm{\dot x}_{2p}\le\Norm{\dot x}_\infty\le C$
this proves the Lipschitz estimate for term eight.
Note that $F_1\equiv 0\equiv F_{14}$.
The estimates for the other eleven $F$-terms
follow similarly. This proves
the Lipschitz Lemma~\ref{le:f}.
\qed
\end{proof}

\subsection{Semigroups and splittings}
\label{sec:splittings}

For any $q\in[1,\infty)$ and $k\in\N_0$
the negative Jacobi operator
$-A:=-A_x$ on $Z:=W^{k,q}$ with dense domain $W^{k+2,q}$
and given by~(\ref{eq:Jacobi-1})
is sectorial and therefore generates
the strongly continuous semigroup
$e^{-sA}\in\Ll(Z)$ given by
\begin{equation}\label{eq:semigroup}
     e^{-sA}:=\frac{1}{2\pi i}
     \int_{\gamma} e^{s\lambda} R(\lambda,-A)
     \: d\lambda
     ,\qquad
     \forall s>0,
\end{equation}
and by $e^{0A}:=\1_Z$ for $s=0$.
Here $R$ denotes the resolvent and $\gamma:\R\to\C\cup\{\infty\}$
is a suitable loop inside the resolvent set $\rho(-A)$.
Sectoriality of $-A$ is well known, but a proof
for the periodic domain $S^1$ is hard to find,
unlike for the domain $\R$.
So we provide the details in~\cite{Joa-LOOPSPACE}.
By nondegeneracy of $x$ the operator $-A$ is
{\bf hyperbolic}, that is its spectrum and the
imaginary axis $i\R$ are disjoint.
Pick a counter-clockwise oriented circle
$\gamma^+:S^1\to(0,\infty)\times i\R$
which encloses the positive part
$\{-\lambda_k,\ldots,-\lambda_1\}$
of the spectrum of $-A$.
The linear operators
\begin{equation}\label{eq:spec-proj}
     {\pi_-}
     :=\frac{1}{2\pi i}\int_{\gamma^+}
     R(\lambda,-A)\: d\lambda
     ,\qquad
     {\pi_+}:=\1-{\pi_-},
\end{equation}
are elements of $\Ll(Z)$ called {\bf spectral projections},
because $(\pi_\pm)^2=\pi_\pm$.

We collect key facts of semigroup theory.
By boundedness of $\pi_\pm$ the images
\begin{equation}
     Z^\pm:=\range{\pi_\pm},
\end{equation}
are closed (Banach) subspaces.
As a vector
space $Z^-$ is spanned by $k$ eigenfunctions
corresponding to the $k=\IND(x)$ negative eigenvalues
of $A$, in particular $Z^-\subset C^\infty(S^1,x^*TM)$.
In contrast $Z^+$ is the
$W^{k,q}$ closure of the sum of eigenspaces
coresponding to positive eigenvectors of $A$.
Thus $Z^+=Z^+(k,q)$.
The obvious identity $\pi_-\pi_+=\pi_+\pi_-=0$
shows that $$
     Z\simeq Z^-\oplus Z^+.
$$
Moreover, this {\bf splitting} is preserved by $A$
and the restrictions of $A$ to the Banach
subspaces $Z^\pm$ are denoted by $A^\pm$.
Since the semigroup $e^{-sA}$ preserves both subspaces
$Z^\pm$, the restrictions $e^{-sA}|_{Z^\pm}$ are
semigroups as well. They are called {\bf subspace semigroups}.
On the other hand, the restrictions $-A^\pm$
themselves are sectorial operators on the Banach
spaces $Z^\pm$ with dense domains $Z^\pm\cap D(A)$.
Therefore they generate strongly continuous
semigroups $e^{-sA^\pm}$ on $Z^\pm$. But these
coincide with the subspace semigroups 
due to the resolvent identity
$R(\lambda,-A)|_{Z^\pm}=R(\lambda,-A^\pm)$
which holds for every $\lambda$ in the resolvent set 
$\rho(-A)\subset\rho(-A^\pm)$.
The upshot is the formula
$$
     e^{-sA}
     =e^{-sA^-}\oplus e^{-sA^+}
     ,\qquad
     s\ge 0.
$$
Note that $D(A)\cap Z^-=Z^-$ by smoothness.
Thus $A^-\in\Ll(Z^-)$ and the series
\begin{equation}\label{eq:exponential}
     e^{-sA^-}
     :=\sum_{k=0}^\infty
     \frac{(-sA^-)^k}{k!},
     \qquad \forall s\in\R,
\end{equation}
is well defined providing a norm continuous group
which for $s\ge 0$ coincides with the subspace semigroup
$e^{-sA}|_{Z^-}$. For \emph{negative} times $s\le 0$
it decays exponentially
$\norm{e^{-sA^-}}_{\Ll(Z^-)}\le ce^{-s\lambda_k}\le ce^{s\mu}$.
The constructions above ``commute''
with Sobolev embeddings
$W^{\ell,r}\hookrightarrow W^{k,q}$.
For $\pi_\pm$ this is again a consequence of
a resolvent identity.

\begin{proposition}
\label{prop:semigroup-NEW}
   Fix integers $\ell\ge k\ge0$ and constants $r\ge q\ge1$.
   Consider the Jacobi operator $A:=A_x$ on
   $Z:=W^{k,q}$ with
   dense domain $W^{k+2,q}$ and its restrictions $A^\pm$ to the
   closed subspaces $Z^\pm:=\pi_\pm(Z)$.
   Fix $\mu>0$ in the spectral gap~(\ref{eq:mu}) of $A$.
   Then there is a constant $c=c(\ell,k,r,q,\mu)$
   such that
\begin{enumerate}

 \item[\rm (a)]
   The operator $-A$ on $L^1$ generates the strongly
   continuous semigroup $e^{-sA}$ on $L^1$
   given by~(\ref{eq:semigroup}).
   Both commute with the spectral projections
   $\pi_\pm$ in~(\ref{eq:spec-proj}).

\item[\rm (b)]
   The subspace semigroup $e^{-sA}|$ on $Z\subset L^1$
   coincides with the strongly continuous
   semigroup $e^{-sA|}$ generated
   by the restriction $-A|$ of $-A$ to $Z\subset L^1$.
   Restricting the semigroup and restricting $A$ both
   commute with the spectral projections
   which themselves satisfy $\pi_\pm(-A|)=\pi_\pm(-A)|$.
  To simplify notation we omit from now on
  the slash sign $\mid$.

 \item[\rm (c)]
   The restriction of $-A$ to $Z^-$
   generates the norm continuous group
   $e^{-sA^-}$ on $Z^-$ given
   by the exponential series~(\ref{eq:exponential}).
   For positive times this group is equal to
   the restriction of the semigroup $e^{-sA}$ to $Z^-$.
   For negative times it holds that
   \begin{equation} \label{eq:c-NEW}
     \big\|e^{-sA^-}\pi_-\big\|_{\Ll(W^{k,q},W^{\ell,r})}
     \le ce^{s\mu}
     ,\quad
     s\le 0.
   \end{equation}

 \item[\rm (d)]
   Restricting $e^{-sA}$ to $Z^+$ gives a
   strongly continuous semigroup on $Z^+$ and
   \begin{equation}\label{eq:reg-for-sing-NEW}
        \Norm{e^{-sA}\pi_+}_{\Ll(W^{k,q},W^{\ell,r})}
        \le c
        s^{-\frac{1}{2}(\frac{1}{q}-\frac{1}{r}+\ell-k)}
        e^{-s\mu}
     ,\quad
     s>0.
   \end{equation}
\end{enumerate}
\end{proposition}

\subsection{The representation formula}

\begin{proposition}\label{prop:representation-formula}
    Consider the nonlinearity $f:X\supset \Uu\to Y$ 
    given by~(\ref{eq:f})
    and the constant $\rho_0$ provided 
    by the Lipschitz Lemma~\ref{le:f}.
    Pick $T>0$ and
    assume $\xi:[0,T]\to X$ is a
    map bounded by $\rho_0$
    thus taking values in $\Uu$.
    Then the following are equivalent. 
\begin{enumerate}
  \item[\rm (a)] 
    The map $\xi:[0,T]\to Y$ 
    is the (unique) solution 
    of the Cauchy 
    problem~(\ref{eq:cauchy-local})
    with initial value $\xi(0)$.

  \item[\rm (b)]
    The map 
    $\xi:(0,T]\to X$ is continuous~\footnote{hence $f\circ\xi:(0,T]\to Y$ is continuous and, by the Lipschitz Lemma~\ref{le:f}, bounded.}
    and satisfies the integral equation,
    also called {\bf representation formula}, given by
    \begin{equation}\label{eq:representation-formula}
    \begin{split}
      \xi(s) 
     &=e^{-sA}\pi_+ \xi(0)
      +\int_0^s e^{-(s-\sigma)A}
      \pi_+f(\xi(\sigma))\, d\sigma
      \\
     &\quad
      +e^{-(s-T)A^-}\pi_- \xi(T)
      -\int_s^T e^{-(s-\sigma)A^-}
      \pi_-f(\xi(\sigma))\, d\sigma
    \end{split}
    \end{equation}
    for every $s\in[0,T]$.
    In the limit $T\to\infty$ the first term
    in line two disappears.
\end{enumerate}
\end{proposition}

\subsection{Local stable manifold theorem}\label{sec:loc-stab-mf}

\begin{theorem}[$C^1$~graph]
\label{thm:local-stable-manifold-graph}
Assume the local setup of
Hypothesis~\ref{hyp:loc-coord-lambda};
see Figure~\ref{fig:fig-local-setup}.
In particular, consider the hyperbolic fixed point
$0$ of the local semi-flow $\phi$ defined
by~(\ref{eq:cauchy-local}) on $\Uu\subset X$.
Then the following is true.
There is a closed ball $\Bb^+\subset X^+$ 
of radius $r>0$ about $0$ such
that a neighborhood of $0$ in the
{\bf local stable manifold}
\begin{equation}\label{eq:loc-stab-mf}
     W^s(0,\Uu)
     :=\left\{ z\in \Uu\,\big|\,
     \text{$\phi(s,z)\in \Uu$ $\forall s>0$ and
     $\lim_{s\to\infty}\phi(s,z)=0$}
     \right\}
\end{equation}
is a graph over $\Bb^+$,
tangent to $X^+$ at $0$.
In fact, there is a Lipschitz
continuous map
$$
     \Gg^\infty
     =\left( G,id \right)
     :\Bb^+\to X^-\times \Bb^+
     ,\qquad
     G(0)=0
     ,\qquad
     dG(0)=0,
$$
of class $C^1$
such that $\Gg^\infty(\Bb^+)$ is a
neighborhood of $0$ in $W^s(0,\Uu)$;
cf. Figure~\ref{fig:fig-lambda-Lemma}.
\end{theorem}

\begin{proposition}[$L^2$~extension]
\label{prop:L2-extension-stab-mf}
Assume
Theorem~\ref{thm:local-stable-manifold-graph}.
Then the linearization 
$d\Gg^\infty(z_+):X^+\to X$ 
extends to a bounded linear operator
on the $L^2$ completions,
uniformly in $z_+\in\Bb^+$.
More precisely, it holds that
\begin{equation*}
     \Norm{d\Gg^\infty(z_+) v}_2
     \le 2 \Norm{v}_2
     ,\qquad
     \Norm{d\Gg^\infty(z_+) v-v}_2
     \le \frac{1}{4} \Norm{v}_2,
\end{equation*}
for all $v\in\pi_+(L^2)$ and $z_+\in \Bb^+$.
\end{proposition}

The local stable manifold
Theorem~\ref{thm:local-stable-manifold-graph}
is well known; see
e.g.~\cite[Thm.~5.2.1]{Henry-81-GeomTheory}
for a proof by the contraction method.
In finite dimensions the theorem is also called
Hadamard-Perron Theorem~\cite{Hadamard-01,Perron-28}.
Observe that proofs of
Theorem~\ref{thm:local-stable-manifold-graph}
and Proposition~\ref{prop:L2-extension-stab-mf}
arise as special cases of the proofs
in Section~\ref{sec:proofs}, formally set $T=\infty$.
Now we recall the {\bf contraction method}
for the stable manifold theorem.
Pick a value for each parameter of interest,
in our case $z_+\in X^+$. Our object of interest
is a heat flow line $\eta:[0,\infty)\to\Uu$
whose initial value projects to $z_+$ under $\pi_+$
and which converges to $0$, as $s\to\infty$.
Find a complete metric space, namely
\begin{equation}\label{eq:exp norm}
\begin{split}
     Z=Z^\rho_\mu
     :=\Bigl\{\eta\in  
     C^0([0,\infty),X)\,\,\Big|\,\,
     \Norm{\eta}_{\exp}
     :=\sup_{s\ge 0} e^{s\frac{\mu}{2}}
     \Norm{\eta(s)}_X
     \le\rho
     \Bigr\}
\end{split}
\end{equation}
for suitable constants $\rho$ and $\mu$,
and a strict contraction on $Z$, namely
\begin{equation*}
     (\Psi_{z_+}\eta)(s)
     =e^{-sA}z_+
     +\int_0^s e^{-(s-\sigma)A}\pi_+
     f(\eta(\sigma)) d\sigma
     -\int_s^\infty e^{-(s-\sigma)A^-}
     \pi_-f(\eta(\sigma)) d\sigma
\end{equation*}
such that the (unique) fixed point $\eta_{z_+}$
is the initial object of interest.
Use the representation
formula~(\ref{eq:representation-formula})
for $T=\infty$ to see that this is indeed true.
In fact this setup works for any
$\mu$ in the spectral gap~(\ref{eq:mu}) of the
Jacobi operator $A:=A_x$,
see~\cite{Joa-LOOPSPACE}, and for all
$\rho>0$ sufficiently small,
that is whenever~(\ref{eq:rho-backward}) holds.
The map
\begin{equation*}
\begin{split}
     G: \Bb^+
     &\to X^-,
      \qquad\qquad\quad
      \Bb^+
      :=\left\{z_+\in X^+:
      \Norm{z_+}_X\le\rho/2c
      \right\},
   \\
     z_+
     &\mapsto \pi_-\left( \eta_{z_+}(0)\right)
\end{split}
\end{equation*}
has the properties asserted by
Theorem~\ref{thm:local-stable-manifold-graph};
here $c=c(\mu)$ is given by
Proposition~\ref{prop:semigroup-NEW}.

\begin{remark}[Unstable manifold]
\label{rem:unstable-manifold}
The contraction method also serves
to represent the elements of the
{\bf local unstable manifold} $W^u(0,\Uu)$; see
          e.g.~\cite[Theorem~5.2.1, proof of 
          Theorem~5.1.3]{Henry-81-GeomTheory}.
By definition this is the set of end points
of all (backward) heat flow lines $\tilde\eta$ in $\Uu$
parametrized by $(-\infty,0]$
and emanating at time $-\infty$ from $0$.
There is a ball $\Bb^-\subset W^u\subset X^-$
of sufficiently small radius $r>0$
such that the following is true.

Pick $\gamma\in\Bb^-$ and consider the backward heat
flow trajectory $\tilde\eta$ which satisfies
$\gamma=\tilde\eta(0)=\pi_-\tilde\eta(0)$.
        (Backward flow invariance of
        descending disks and
        $W^u\subset X^-$ imply that
        $\tilde\eta$ lies completely in $X^-$.)
Note that $\tilde\eta$ is asymptotic to zero in
infinite backward time since $\tilde\eta(0)$ lies in a
descending disk. 
By~(\ref{eq:representation-formula})
and the uniqueness Theorem~17
in~\cite{Joa-HEATMORSE-II} for
action bounded backward heat flow solutions
$\tilde\eta$ is equal to the unique fixed
point of the map
\begin{equation*}
\begin{split}
     \left(\Phi_{\gamma}\tilde\eta\right)(s)
     :=e^{-sA^-}\gamma
    &-\int_s^0 e^{-(s-\sigma)A^-}\pi_-
     f(\tilde\eta(\sigma)) d\sigma
     +\int_{-\infty}^s e^{-(s-\sigma)A}
     \pi_+ f(\tilde\eta(\sigma)) d\sigma
\end{split}
\end{equation*}
which acts as a strict contraction
on the complete metric space
\begin{equation*}
\begin{split}
     \tilde Z=\tilde Z_{\frac{\mu}{2},\rho}
     :=\Bigl\{\tilde\eta\in  
     C^0((-\infty,0],X)\,\,\Big|\,\,
     \Norm{\tilde\eta}_{\exp}
     :=\sup_{s\le 0} e^{-s\frac{\mu}{2}}
     \Norm{\tilde\eta(s)}_X
     \le\rho
     \Bigr\}.
\end{split}
\end{equation*}
\end{remark}

\section{Proofs}\label{sec:proofs}

\subsection{Proof of the backward $\lambda$-Lemma (Theorem~\ref{thm:backward-lambda-Lemma})}

\begin{figure}
  \includegraphics{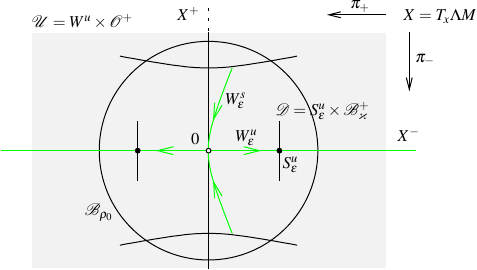}
  \caption{Local setup}
  \label{fig:fig-local-setup}
\end{figure}

Uniform exponential 
convergence in step~6 is the heart
of the proof. It relies on a
suitable time decomposition of trajectories.
Throughout assume
\begin{hypothesis}\label{hyp:backward-lambda-Lemma}
Assume the local setup of
Hypothesis~\ref{hyp:loc-coord-lambda}; see
Figure~\ref{fig:fig-local-setup}.
Consider the constants $\rho_0>0$ and $\kappa_*\ge 1$
and the continuous function
$\kappa(\rho)$ with $\kappa(0)=0$ provided
by the Lipschitz Lemma~\ref{le:f}.
Note that $\Bb_{\rho_0}\subset\Uu$.
Fix $\mu$ in the spectral
gap $(0,d)$ given by~(\ref{eq:mu}) and
a constant $c=c(\mu)\ge 1$ satisfying
Proposition~\ref{prop:semigroup-NEW} for
the (finitely many) choices of $(\ell,k,r,q)$
that will be used in the present proof.
We may assume that~\footnote{
  Otherwise, choose $\rho_0>0$ smaller. This
  leads to a smaller $\eps$ in
  Hypothesis~\ref{hyp:loc-coord-lambda}~(a).
  Condition~(\ref{eq:rho_0-backward}) is used in step~4
  and in the proof of
  Theorem~\ref{thm:uniform-C1},
  both concerning $C^1$.
}
\begin{equation}\label{eq:rho_0-backward}
     \rho_0 c^2\kappa_*
     \left(
     \frac{1}{\mu}
     +\frac{8}{\mu^\frac{1}{4}}
     +\frac{6}{\mu^\frac{5}{4}}
     \right)
     \le\frac{1}{8}.
\end{equation}
In Hypothesis~\ref{hyp:loc-coord-lambda}~(b)
we picked $\varkappa\in(0,\rho_0)$ and so the
constant~\footnote{
  The definition of $T_1$ ensures in step~2 the second
  of the two endpoint conditions~(\ref{eq:xi-end}).
  }
\begin{equation}\label{eq:T_1}
     T_1=T_1(\varkappa,\mu,\rho_0)
     :=-\frac{2}{\mu}\ln
     \frac{\varkappa}{\rho_0}> 0
\end{equation}
is well defined.
Assume $\rho\in(0,\frac12\rho_0]$ is sufficiently
small such that
\begin{equation}\label{eq:rho-backward}
     c^2\kappa(\rho)
     \left(\frac{9}{\mu^\frac{1}{4}}
     +\frac{4}{3\mu}
     +\frac{3}{\mu^\frac{5}{4}}
     \right) 
     \le\frac{1}{8}
\end{equation}
and such that all points of $\Uu$ of distance $\le\rho$
to the descending sphere $S^u_\eps$ are
contained in $\Bb_{\rho_0}$.
Fix a constant $T_2=T_2(c,\mu,\rho,\eps)\ge0$
such that $8ce^{-\mu T_2/4}\le1$ and
$\phi_{-T_2/8}W^u_\eps\subset\Bb^-:=\Bb_{\rho/2c}\cap X^-$;
see Remark~\ref{rem:unstable-manifold}.~\footnote{
  The conditions on $T_2$ will be
  used in step~6, in particular in~(\ref{eq:back-traj}).
  }
Set
$
     T_0
     :=\max\{T_1,T_2\}
     > 0
$.
\end{hypothesis}
Pick $T\ge T_0$ and $\gamma\in S^u_\eps$ and consider the
infinite dimensional disk 
$\Dd_\gamma=\{\gamma\}\times\Bb^+_\varkappa$. The key
observation to represent the preimage 
${\phi_T}^{-1}\Dd_\gamma$ under the time-$T$-map $\phi_T$
as a graph over the stable subspace $X^+$ is the
fact that to any pair $(q_u,z_+)\in X^-\oplus X^+$
sufficiently close to zero there corresponds a unique
heat flow trajectory $\xi$ whose initial value $\xi(0)$
projects under $\pi_+$ to $z_+$ and whose endpoint at
time $T$ projects under $\pi_-$ to $q_u$; see
Remark~\ref{rem:mixed-Cauchy}. In particular, for
$q_u:=\gamma$ any $z_+\in X^+$ near the origin
corresponds to a unique heat flow line 
$\xi=\xi^T_{\gamma,z_+}$ which ends at time $T$ in
$\Dd_\gamma$. Because its initial value $\xi(0)$ is of
the form $\left(\pi_-\xi(0),z_+\right)$, it is natural
to define the map $G^T_\gamma(z_+):=\pi_-\xi(0)$
whose graph at $z_+$ reproduces $\xi(0)$.
In fact, we prove that for any $z_+\in X^+$ with
$\norm{z_+}\le\rho/2c$ there is precisely one semi-flow
line $\xi=\xi_{\gamma,z_+}^T$ with initial
condition $\pi_+ \xi(0)=z_+$ and endpoint condition
$\xi(T)\in\Dd_\gamma$. The latter is equivalent to
\begin{equation}\label{eq:xi-end}
     \pi_-\xi(T)=\gamma
     \quad\wedge\quad
     \Norm{\xi(T)-\gamma}_X\le\varkappa.
\end{equation} 
We will see in step~2 that the definition of $T_1$ assures the second condition.

The key step to determine the unique
semi-flow line $\xi$ associated
to $(T,\gamma,z_+)$
is to set up a strict
contraction on a complete
metric space $Z^T$ whose (unique) fixed
point is $\xi$.
Set
\begin{equation}\label{eq:exp-T}
     \Norm{\xi}_{\exp}
     =\Norm{\xi}_{\exp,T}
     :=\max_{s\in[0,T]} 
     e^{s\frac{\mu}{2}}
     \Norm{\xi(s)}_X
\end{equation}
and for $\gamma_T:=\phi_{-T}\gamma$ define
\begin{equation}\label{eq:ZT}
     Z^T
     =Z^{T,\gamma}_{\mu/2,\rho}
     :=\left\{
     \xi\in
     C^0([0,T],X)\,\colon
     \Norm{\xi-\phi_\cdot\gamma_T}
     _{\exp}
     \le\rho\right\}.
\end{equation}
Consider the map
$\Psi^T=\Psi^T_{\gamma,z_+}$ defined on $Z^T$ by
\begin{equation}\label{eq:Psi}
\begin{split}
     \left(\Psi_{\gamma,z_+}^T
     \xi\right)(s)
    &:=e^{-sA}z_+
     +\int_0^s e^{-(s-\sigma)A}\pi_+
     f(\xi(\sigma))
     \, d\sigma
     \\
    &\quad
     +e^{-(s-T)A^-}\gamma
     -\int_s^T e^{-(s-\sigma)A^-}\pi_-
     f(\xi(\sigma))
     \, d\sigma
\end{split}
\end{equation}
for every $s\in[0,T]$.
The fixed points of $\Psi^T$ 
correspond to the desired heat flow
trajectories by
Proposition~\ref{prop:representation-formula}.
By step~1 and step~2 below $\Psi^T$
is a strict contraction on $Z^T$.
Hence by the Banach fixed point theorem
it admits a unique fixed point $\xi_{\gamma,z_+}^T$
and for $\Bb^+:=\Bb^+_{\rho/2c}\subset X^+$
we define the map
\begin{equation}\label{eq:G^T}
     G^T:S^u_\eps\times
     \Bb^+\to W^u\subset X^-
     ,\qquad
     \left(\gamma,z_+\right)\mapsto 
     \pi_-\xi_{\gamma,z_+}^T(0)
     =:G^T_\gamma(z_+).
\end{equation}
Actually $\Bb^+$ is the same ball for
which the stable manifold
Theorem~\ref{thm:local-stable-manifold-graph}
holds true~\cite{Joa-LOOPSPACE}.

The proof takes six steps.
Fix $\gamma\in S^u_\eps$ and $z_+\in\Bb^+$
and abbreviate $\Psi^T=\Psi_{\gamma,z_+}^T$.

\vspace{.2cm}
\noindent
{\bf Step 1.}
{\it Fix $T\ge0$.
Then the set $Z^T$ equipped with the
metric induced by the exp norm is a
complete metric space, any 
$\xi\in Z^T$ takes values in 
$\Bb_{\rho_0}$, and $\Psi^T$ acts on $Z^T$.
}

\begin{proof}
In case of
the compact domain
$[0,T]$ the space $C^0([0,T],X)$
is complete with respect to the
supremum norm, hence with respect 
to the exp norm as both norms
are equivalent by compactness of 
$[0,T]$. The subset $Z^T\subset
C^0([0,T],X)$ is
closed with respect to the exp norm.
By the assumption
which immediately follows~(\ref{eq:rho-backward})
the elements of $Z^T$ take values 
in $\Bb_{\rho_0}$, hence in $\Uu$.

To see that $\Psi^T$ acts on $Z^T$
we need to verify that $\Psi^T\xi$ is continuous
and satisfies the 
exponential decay condition in~(\ref{eq:ZT})
whenever $\xi\in Z^T$.
By definition
$\Psi^T\xi$ is a sum of four terms.
That each one is continuous
as a map $[0,T]\to X$ is standard.
For terms one, two, and four
see step~1~(iii) in the proof of
Theorem~\ref{thm:local-stable-manifold-graph} 
given in~\cite{Joa-LOOPSPACE}.
For term three continuity follows 
from the definition
of the exponential by the
power series~(\ref{eq:exponential}).
For latter reference we sketch the argument for term two
which we denote by $F(s)$:
Continuity of $F:[0,T]\to X$ and the fact
that $F(0)=0$ (used in steps~2 and~3 below) both 
follow by an analogue
of~\cite[Le.~9.7~a)]{Joa-HABILITATION} for $-A$
instead of $\Delta$ and with $p=2$;
see also~\cite[Le.~3.2.1]{Henry-81-GeomTheory}.
The condition to 
be checked is that the map $\tilde{f}
:=\pi_+ \circ f\circ\xi:[0,T]\to Y^+\hookrightarrow Y$
is continuous and bounded: This is true since
$\xi:[0,T]\to X$ is continuous and bounded by 
definition of $Z^T$ and so is $f$
by 
Lemma~\ref{le:f}.

We prove exponential decay.
For $s\in[0,T]$
consider the heat flow trajectory
given by $\phi_s\gamma_T$.
By the representation formula of
Proposition~\ref{prop:representation-formula}
it satisfies 
\begin{equation}\label{eq:phi-s}
\begin{split}
     \phi_s\gamma_T
    &=\int_0^s e^{-(s-\sigma)A}\pi_+
     f(\phi_\sigma \gamma_T)\, d\sigma
     +e^{-(s-T)A^-}\gamma
     -\int_s^T e^{-(s-\sigma)A^-}\pi_-
     f(\phi_\sigma \gamma_T)\, d\sigma.
\end{split}
\end{equation}
Here we used that $\pi_+\gamma_T=0$,
because $\gamma$ and therefore 
$\gamma_T=\phi_{-T}\gamma$ lies in
$W^u(0,\Uu)\subset X^-$ by backward
flow invariance. By the same argument
$\pi_-\phi_T\gamma_T=\pi_-\gamma
=\gamma$.
By definition~(\ref{eq:Psi}) of $\Psi^T$
and~(\ref{eq:phi-s})
we get for $s\in[0,T]$ the estimate
\begin{equation}\label{eq:Psixi-bounded}
\begin{split}
    &\Norm{\left(\Psi^T\xi\right)(s)
     -\phi_s\gamma_T
     }_X
   \\
    &\le
     \Norm{e^{-sA}z_+}_X
     +\int_0^s
     \Norm{e^{-(s-\sigma)A}\pi_+}
     _{\Ll(Y,X)}
     \Norm{f(\xi(\sigma))
     -f(\phi_\sigma\gamma_T)
     }_Y d\sigma
   \\
    &\quad
     +\int_s^T
     \Norm{e^{-(s-\sigma)A^-}\pi_-}
     _{\Ll(Y,X)}
     \Norm{f(\xi(\sigma))
     -f(\phi_\sigma\gamma_T)
     }_Y d\sigma
   \\
    &\le ce^{-s\mu}\Norm{z_+}_X
     +c\kappa(\rho)e^{-s\frac{\mu}{2}}
     \Norm{\xi-\phi_\cdot\gamma_T}
     _{\exp}
     \int_0^s
     \frac{e^{-(s-\sigma)\frac{\mu}{2}}}
     {(s-\sigma)^{\frac{3}{4}}}
     \, d\sigma
   \\
    &\quad
     +c\kappa(\rho)e^{-s\frac{\mu}{2}}
     \Norm{\xi-\phi_\cdot\gamma_T}
     _{\exp}
     \int_s^T 
     e^{(s-\sigma)\frac{3}{2}\mu}
     \, d\sigma
   \\
    &\le 
     \frac{\rho}{2} e^{-s\mu}
     + 
     c\kappa(\rho)
     \left(
     \frac{8}{\mu^{1/4}} 
     +\frac{2}
     {3\mu} 
     \right)
     \rho e^{-s\frac{\mu}{2}}
     \le\rho e^{-s\frac{\mu}{2}}
\end{split}
\end{equation}
where the last inequality is by
smallness~(\ref{eq:rho-backward}) of $\rho$.
Inequality two follows by the exponential decay 
Proposition~\ref{prop:semigroup-NEW}
with constant $c$
and the Lipschitz Lemma~\ref{le:f}
for $f$ with Lipschitz constant $\kappa(\rho)$.
We multiplied the integrands by
$e^{-\sigma\frac{\mu}{2}}
e^{\sigma\frac{\mu}{2}}$ to create the exp norms.
Inequality three uses
$\norm{z_+}_X\le\frac{\rho}{2c}$ and boundedness
of the exp norms by $\rho$ since $\xi\in Z^T$.
We also used that
\begin{equation}\label{eq:int-preGamma}
     \int_s^\infty
     e^{(s-\sigma)\frac{3}{2}\mu}
     \, d\sigma
     =\frac{2}{3\mu}.
\end{equation}
To estimate the other integral define
$\Gamma(\alpha):=\int_0^\infty e^{-\tau}\tau^{\alpha-1}\, d\tau$
for $\alpha>0$.
The $\Gamma$ function
satisfies $\Gamma(\frac{1}{4})
=4\Gamma(\frac{5}{4})\le 4$ since
$\Gamma\le1$ on the interval $[1,2]$. Hence
\begin{equation}\label{eq:int-Gamma}
     \int_0^s\frac{e^{-(s-\sigma)\frac{\mu}{2}}}
     {(s-\sigma)^{\frac{3}{4}}} \, d\sigma
     =\left(\frac{2}{\mu}\right)^{\frac{1}{4}}
     \int_0^{s\frac{\mu}{2}} e^{-\tau} 
     \tau^{\frac{1}{4}-1} \, d\tau
     \le \frac{2^{\frac{1}{4}}}{\mu^{\frac{1}{4}}}\,
     \Gamma(\tfrac{1}{4})
     \le \frac{8}{\mu^{\frac{1}{4}}}.
\end{equation}
\qed
\end{proof}

\noindent
{\bf Step 2.}
{\it For $T\ge0$ the map 
$\Psi^T$ acts as a strict contraction on $Z^T$.
Each image point $\Psi^T\xi$
satisfies the initial condition
$\pi_+\left(\Psi^T\xi\right)(0)=z_+$
and, if $T\ge T_1$, also the endpoint 
condition~(\ref{eq:xi-end}),
that is $\left(\Psi^T\xi\right)(T)
\in\Dd_\gamma=\{\gamma\}\times \Bb^+_\varkappa$.
}

\begin{proof}
Assume $T\ge 0$ and fix 
$\xi_1,\xi_2\in Z^T$.
Similarly to~(\ref{eq:Psixi-bounded})
we obtain that
\begin{equation}\label{eq:est-step-2}
\begin{split}
    &\Norm{\left(\Psi^T\xi_1\right)(s)
     -\left(\Psi^T\xi_2\right)(s)}_X
   \\
     &\le
      \int_0^s\Norm{e^{-(s-\sigma)A}
      \pi_+}_{\Ll(Y,X)}
      \Norm{f(\xi_1(\sigma))
      -f(\xi_2(\sigma))}_Y d\sigma
   \\
     &\quad
      +\int_s^T
      \Norm{e^{-(s-\sigma)A^-}\pi_-}
      _{\Ll(Y,X)}
      \Norm{f(\xi_1(\sigma))
      -f(\xi_2(\sigma))}_Y d\sigma
    \\
      &\le c\kappa(\rho)
       \left(
       \frac{8}{\mu^{1/4}}
       +\frac{2}{3\mu} 
       \right)
       e^{-s\frac{\mu}{2}}
       \Norm{\xi_1-\xi_2}_{\exp}
\end{split}
\end{equation}
for every $s\in[0,T]$. Now use the smallness
assumption~(\ref{eq:rho-backward}) on
$\rho$ to conclude that
$\norm{\Psi^T\xi_1-\Psi^T\xi_2}_{\exp}
\le\frac{1}{2}\norm{\xi_1-\xi_2}_{\exp}$.

The identities
$\pi_+\left(\Psi^T\xi\right)(0)=z_+$ and
$\pi_-\left(\Psi^T\xi\right)(T)=\gamma$
follow from definition~(\ref{eq:Psi}) of $\Psi^T$,
the identities $\pi_+\pi_-=\pi_-\pi_+=0$, strong
continuity of the semigroups on $X^-$ and $X^+$
asserted by
Proposition~\ref{prop:semigroup-NEW},
continuity and boundedness of both integrands,
and $F(0)=0$ by the proof of step~1.
Concerning the second endpoint
condition in~(\ref{eq:xi-end})
assume $T\ge T_1$ and
evaluate~(\ref{eq:Psixi-bounded})
at $s=T$ to get
$$
     \Norm{\left(\Psi\xi\right)(T)
     -\gamma}_X
     \le\rho
     e^{-T\frac{\mu}{2}}
     \le e^{-T_1\frac{\mu}{2}}\rho_0 
     =\varkappa
$$
where the last step is by 
definition of $T_1$ in~(\ref{eq:T_1}).
\qed
\end{proof}

\noindent
{\bf Step 3.}
{\it For $T\ge0$ the map $G^T:S^u_\eps\times\Bb^+\to X^-$
defined by~(\ref{eq:G^T}) is of class $C^1$ and,
for each $\gamma\in S^u_\eps$, the map 
$G^T_\gamma:=G^T(\gamma,\cdot):\Bb^+\to X^-$ satisfies
$$
     G^T_\gamma(0)
     =\phi_{-T}\gamma
     =:\gamma_T
     ,\qquad
     \mathrm{graph}\, G^T_\gamma
     =\left\{\xi_{\gamma,z_+}^T(0)
     \,\big|\, 
     z_+\in\Bb^+\right\}.
$$
} 

\begin{proof}
Assume $T\ge 0$. By step~2 and its proof the map
$$
     \Psi^T:
     S^u_\eps\times\Bb^+\times Z^T\to Z^T,\qquad
     (\gamma,z_+,\xi)\mapsto\Psi^T_{\gamma,z_+}\xi
$$
is a uniform contraction on $Z^T$ with contraction
factor $\tfrac12$. (Actually $Z^T$ depends on $\gamma$,
but the complete metric spaces associated to different
$\gamma$'s are quasi-isometric.) Observe that
$\Psi^T$ is linear, hence smooth, in $\gamma$ and in
$z_+$ and of class $C^1$ in $\xi$, because $f$ is of
class $C^1$ by the Lipschitz Lemma~\ref{le:f}.
Hence by the uniform contraction principle,
see e.g.~\cite{Chow-Hale-Bifurcation-Theory}, the map 
$\lambda:S^u_\eps\times \Bb^+\to Z^T$ assigning to
$(\gamma,z_+)$  the unique fixed point
$\xi^T_{\gamma,z_+}$ of $\Psi^T_{\gamma,z_+}$ is of class
$C^1$. So is its composition with (linear)
evaluation $ev_0:Z^T\to X$, $\xi\mapsto\xi(0)$,
and (linear) projection $\pi_-:X\to X^-$.
But overall this composition is $G^T$ by 
definition~(\ref{eq:G^T}).
This proves that $G^T$, thus $\Gg$, is of class $C^1$
in $\gamma$ and $z_+$.

Consider the heat flow trajectory
$\tilde\eta:[0,T]\to X$, 
$s\mapsto\phi_s\gamma_T=\phi_{s-T}\gamma$.
By Remark~\ref{rem:unstable-manifold}
it takes values in $X^-$, because
$\tilde\eta(T)=\gamma\in S^u_\eps=\p W^u_\eps\subset W^u$
lies in a descending disk.
Hence $\pi_+\tilde\eta(0)=0$
and $\pi_-$ leaves $\tilde\eta$ pointwise invariant.
An argument as in Remark~\ref{rem:unstable-manifold}
(using likewise forward uniqueness)
shows that $\tilde\eta=\xi^T_{\gamma,0}$.
Thus
$
     G^T_\gamma(0)
     :=\pi_-\xi^T_{\gamma,0}(0)
     =\pi_-\tilde\eta(0)
     =\tilde\eta(0)
     =\gamma_T
$.
To get the desired representation of 
$\mathrm{graph}\, G^T_\gamma$ observe that
\begin{equation}\label{eq:G=xi}
     \Gg^T_\gamma(z_+)
     :=\left( G^T_\gamma(z_+),z_+\right)
     =\left(\pi_-\xi^T_{\gamma,z_+}(0),
     \pi_+ \xi^T_{\gamma,z_+}(0)\right)
     = \xi^T_{\gamma,z_+}(0)
\end{equation}
by definition~(\ref{eq:G^T}).
The first identity also uses the fixed point
property and the initial condition proved in step~2.
The final identity is by $\pi_-\oplus\pi_+=\1_Y$.
\qed
\end{proof}

\noindent
{\bf Step 4.}
{\it The map $\Gg$ is of class $C^1$.
The map $T\mapsto\Gg(T,\gamma,z_+)$ is Lipschitz
continuous and its derivative is locally H\"older
continuous with exponent $\alpha=\frac{1}{8}$.
The map $\gamma\mapsto\Gg(T,\gamma,z_+)$ is
Lipschitz continuous.
} 

\begin{proof}
By step~3 the map $\Gg$ is of class $C^1$ in
the $\gamma$ and $z_+$ variables.
By compactness of the
$(k-1)$-dimensional sphere
$S^u_\eps$, the derivative of $\Gg$ with
respect to $\gamma$ is bounded. Thus
$\Gg$ is Lipschitz continuous in $\gamma$.

{\it We prove that $\Gg$ is Lipschitz continuous in $T$.}
Fix $T\ge T_0>0$,
$\gamma\in S^u_\eps$, and $z_+\in\Bb^+$.
The fixed point $\xi^T:=\xi^T_{\gamma,z_+}$
of $\Psi^T$ is given by~(\ref{eq:Psi}) and
the one of $\Psi^{T+\tau}$ by
\begin{equation*}
\begin{split}
     \xi^{T+\tau}(s)
     :=\xi^{T+\tau}_{\gamma,z_+}(s)
    &=e^{-sA}z_+
     +\int_0^s e^{-(s-\sigma)A}\pi_+
     f(\xi^{T+\tau}(\sigma))
     \, d\sigma
     \\
    &\quad
     +e^{-(s-T-\tau)A^-}\gamma
     -\int_s^{T+\tau} e^{-(s-\sigma)A^-}\pi_-
     f(\xi^{T+\tau}(\sigma))
     \, d\sigma.
\end{split}
\end{equation*}
For $s\in[0,T]$ and $\tau\ge0$ we obtain,
analogously to~(\ref{eq:Psixi-bounded}), the estimate
\begin{equation*}
\begin{split}
    &\Norm{\xi^{T+\tau}(s)-\xi^T(s)}_X
   \\
    &\le\int_0^s\bigl\| e^{-(s-\sigma)A}\pi_+\bigr\|
     _{\Ll(Y,X)}
     \Norm{f(\xi^{T+\tau}(\sigma))-f(\xi^T(\sigma))}_Y
     \, d\sigma
     +\bigl\| \bigl(e^{\tau A^-}-\1\bigr)
     e^{-(s-T)A^-}\gamma\bigr\|_X
   \\
    &\quad
     +\int_s^T\bigl\| e^{-(s-\sigma)A^-}\pi_-\bigr\|
     _{\Ll(Y,X)}
     \Norm{f(\xi^{T+\tau}(\sigma))-f(\xi^T(\sigma))}_Y
     \, d\sigma
   \\
    &\quad
     +\int_T^{T+\tau}
     \bigl\| e^{-(s-\sigma)A^-}\pi_-\bigr\|_{\Ll(Y,X)}
     \Norm{f(\xi^{T+\tau}(\sigma))}_Y
     \, d\sigma
   \\
    &\le c\kappa(\rho)
     \Norm{\xi^{T+\tau}-\xi^T}_{C^0([0,T],X)}
     \left(
     \int_0^s\frac{e^{-(s-\sigma)\mu}}
     {(s-\sigma)^{\frac{3}{4}}}\, d\sigma
     +\int_s^T e^{(s-\sigma)\mu}\, d\sigma
     \right)
   \\
    &\quad
     +\tau c\abs{\lambda_1}
     \cdot ce^{(s-T)\mu}
     \Norm{\gamma}_X
     +c\kappa(\rho)\rho_0
     \int_T^{T+\tau} e^{(s-\sigma)\mu}
     \, d\sigma
   \\
    &\le c\kappa(\rho)
     \left(\frac{8}{\mu^{1/4}}
     +\frac{1}{\mu}\right)
     \Norm{\xi^{T+\tau}-\xi^T}_{C^0([0,T],X)}
     +\tau\rho_0 c^2\abs{\lambda_1} e^{(s-T)\mu}
   \\
    &\quad
     + c\kappa(\rho)\rho_0\frac{e^{(s-T)\mu}}{\mu}
     \left(1-e^{-\tau\mu}\right)
   \\
    &\le\frac{1}{8}
     \Norm{\xi^{T+\tau}-\xi^T}_{C^0([0,T],X)}
     +\tau\rho_0\left(c^2\abs{\lambda_1}+1\right)
     e^{(s-T)\mu}.
\end{split}
\end{equation*}
Inequality two uses the Lipschitz
Lemma~\ref{le:f} for $f$ and the exponential
estimates of Proposition~\ref{prop:semigroup-NEW}.
To estimate the second of the four terms recall that
$X^-$ is spanned by an orthonormal basis of eigenvectors
of $A^-\in\Ll(X^-)$ corresponding to the
eigenvalues $\lambda_1\le\ldots\le\lambda_k<0$.
Hence $\norm{A^-}=-\lambda_1=\abs{\lambda_1}$.
Since $e^{-s\mu}\le 1$ we get that
\begin{equation}\label{eq:1-e}
     \frac{1-e^{-\mu\tau}}{\mu}
     =\int_0^\tau e^{-s\mu}\; ds
     \le\tau.
\end{equation}
Thus~\cite[Prop.~1.3.6.~(ii)]{LUNARDI-InetSem}
implies the estimate
\begin{equation}\label{eq:e-1}
     \norm{e^{\tau A^-}-\1}
     _{\Ll(X^-)}
     =\Norm{A^-\int_0^\tau e^{\sigma A^-}\, d\sigma}
     _{\Ll(X^-)}
     \le\abs{\lambda_1}
     \int_0^\tau ce^{-\sigma\mu}\, d\sigma
     \le\tau c\abs{\lambda_1}.
\end{equation}
Coming back to inequality two above,
we used that $\xi^{T+\tau}\in Z^{T+\tau}$ in term four
takes values in $\Bb_{\rho_0}$ by step~1.
Inequality three uses~(\ref{eq:int-Gamma})
and~(\ref{eq:int-preGamma}) for the first two
integrals and that
$\gamma\in S^u_\eps\subset\Bb_{\rho_0}$.
Inequality four uses~(\ref{eq:1-e}) and
smallness~(\ref{eq:rho-backward})
of $\rho$ which also implies that
$c\kappa(\rho)\le1$.
Now set $c_1=2(c^2\abs{\lambda_1}+1)$
and take the supremum over $s\in[0,T]$ to
get that
\begin{equation}\label{eq:xi-xi}
     \Norm{\xi^{T+\tau}-\xi^T}_{C^0([0,T],X)}
     \le \rho_0 c_1\tau.
\end{equation}
Therefore by~(\ref{eq:G=xi}) we get
$
     \norm{\Gg^{T+\tau}_\gamma(z_+)-\Gg^T_\gamma(z_+)}_X
     =\norm{\xi^{T+\tau}(0)-\xi^T(0)}_X
     \le \rho_0 c_1\tau
$
and this proves that
$\Gg(T,\gamma,z_+)=\Gg^T_\gamma(z_+)$ is Lipschitz
continuous in $T$. The difference $\xi^{T+\tau}-\xi^T$
is illustrated by Figure~\ref{fig:fig-xi-xi}.
\begin{figure}
  \includegraphics{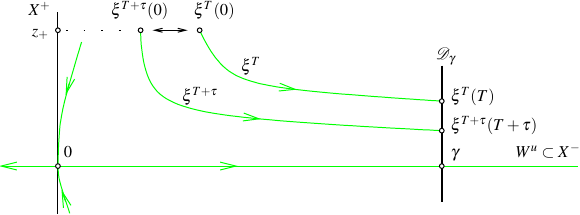}
  \caption{The difference 
           $\Gg^{T+\tau}_\gamma(z_+)-\Gg^T_\gamma(z_+)$}
  \label{fig:fig-xi-xi}
\end{figure}

{\it We prove that 
$T\mapsto\frac{d}{dT}\Gg(T,\gamma,z_+)$
is locally H\"older.}
Consider the derivative
\begin{equation*}
\begin{split}
     \Theta^T(s)
    &:=\left.\tfrac{d}{d\tau}\right|_{\tau=0}
     \xi^{T+\tau}_{\gamma,z_+}(s)
   \\
    &=\int_0^s e^{-(s-\sigma)A}\pi_+
     \bigl(df|_{\xi^T(\sigma)}\circ \Theta^T(\sigma)\bigr)
     d\sigma
     \;+\; A^- e^{-(s-T)A^-}\gamma
     \\
    &\quad
     -e^{-(s-T)A^-}\pi_- f(\xi^T(T))
     -\int_s^T e^{-(s-\sigma)A^-}\pi_-
     \Bigl(df|_{\xi^T(\sigma)}\circ \Theta^T(\sigma)\Bigr)
     d\sigma.
\end{split}
\end{equation*}
Since $\frac{d}{dT}\Gg(T,\gamma,z_+)=\Theta^T(0)$
by~(\ref{eq:G=xi}), it remains to show that the map
$T\mapsto \Theta^T(0)\in X$ is locally H\"older continuous.
By definition of $\Theta^T$ we get the identity
\begin{equation*}
\begin{split}
     \Theta^{T+\tau}(s)-\Theta^T(s)
    &=\int_0^s e^{-(s-\sigma)A}\pi_+
     df|_{\xi^{T+\tau}(\sigma)}
     \left( \Theta^{T+\tau}(\sigma)-\Theta^T(\sigma)\right)
     \, d\sigma
  \\
    &\quad
     +\int_0^s e^{-(s-\sigma)A}\pi_+
     \left(
     df|_{\xi^{T+\tau}(\sigma)}-df|_{\xi^T}(\sigma)
     \right)\circ 
     \Theta^T(\sigma)
     \, d\sigma
   \\
    &\quad
     +\bigl(e^{\tau A^-}-\1\bigr)
     A^- e^{-(s-T)A^-}\gamma
     -\bigl(e^{\tau A^-}-\1\bigr)
     e^{-(s-T)A^-}\pi_- f(\xi^{T+\tau}(T+\tau))
   \\
    &\quad
     -e^{-(s-T)A^-}\pi_-
     \left(
     f(\xi^{T+\tau}(T+\tau))-f(\xi^T(T))
     \right)
   \\
    &\quad
     -\int_s^T e^{-(s-\sigma)A^-}\pi_-
     df|_{\xi^{T+\tau}(\sigma)}
     \left( \Theta^{T+\tau}(\sigma)-\Theta^T(\sigma)\right)
     \, d\sigma
  \\
    &\quad
     -\int_s^T e^{-(s-\sigma)A^-}\pi_-
     \left(
     df|_{\xi^{T+\tau}(\sigma)}-df|_{\xi^T}(\sigma)
     \right)\circ 
     \Theta^T(\sigma)
     \, d\sigma
   \\
    &\quad
     -\int_T^{T+\tau}
     e^{-(s-\sigma)A^-}\pi_-
     df|_{\xi^{T+\tau}(\sigma)}\circ \Theta^{T+\tau}(\sigma)
     \, d\sigma
\end{split}
\end{equation*}
for all $s\in[0,T]$ and $\tau\ge0$.
To obtain terms one and two we added zero, similarly
for terms four and five and terms six and seven.
Abbreviate the norm of the Banach space $C^0([0,T],X)$ by
$\norm{\cdot}_{C^0_T}$
and combine terms one and six and terms two
and seven to get that
\begin{equation*}
\begin{split}
    &\Norm{\Theta^{T+\tau}(s)-\Theta^T(s)}_X
   \\
    &\le c\kappa(\rho)
     \Norm{\Theta^{T+\tau}-\Theta^T}_{C^0_T}
     \left(
     \int_0^s\frac{e^{-(s-\sigma)\mu}}
     {(s-\sigma)^{\frac{3}{4}}}\, d\sigma
     +\int_s^T e^{(s-\sigma)\mu}\, d\sigma
     \right)
   \\
    &\quad
     +c\kappa_*
     \Norm{\xi^{T+\tau}-\xi^T}_{C^0_T}
     \Norm{\Theta^T}_{C^0_T}
     \left(
     \int_0^s\frac{e^{-(s-\sigma)\mu}}
     {(s-\sigma)^{\frac{3}{4}}}\, d\sigma
     +\int_s^T e^{(s-\sigma)\mu}\, d\sigma
     \right)
   \\
    &\quad
     +\tau c^2\abs{\lambda_1}
     \left(
     \abs{\lambda_1}\cdot\Norm{\gamma}_X
     +\kappa(\rho)\Norm{\xi^{T+\tau}(T+\tau)}_X
     \right)
     e^{(s-T)\mu}
   \\
    &\quad
     +c\kappa(\rho)
     \Norm{\xi^{T+\tau}(T+\tau)-\xi^T(T)}_X
     e^{(s-T)\mu}
     +c\kappa(\rho)\Norm{\Theta^{T+\tau}}_{C^0_{T+\tau}}
     \int_T^{T+\tau} e^{(s-\sigma)\mu}
     \, d\sigma
   \\
    &\le c\kappa(\rho)
     \left(\frac{8}{\mu^{1/4}}
     +\frac{1}{\mu}\right)
     \Norm{\Theta^{T+\tau}-\Theta^T}_{C^0_T}
     +\tau c\kappa_*\rho_0^2 c_1^2
     \left(
     \frac{8}{\mu^{1/4}}+\frac{1}{\mu}
     \right)
   \\
    &\quad
     +\tau c\rho_0\abs{\lambda_1}
     \left(c\abs{\lambda_1}+1\right)
     +
     \tau^{\frac{1}{8}}\rho_0\left(
     c_3T^{\frac{3}{8}}+c_4\tau^{\frac{1}{8}}
     \right)
     +\frac{\rho_0 c_1}{\mu}
     \left(1-e^{-\tau\mu}\right)
   \\
    &\le\frac{1}{8}
     \Norm{\Theta^{T+\tau}-\Theta^T}_{C^0_T}
     +\rho_0\left(
     \tau c_1^2
     +\tau c\abs{\lambda_1} c_1
     +
     \tau^{\frac{1}{8}}\left(
     c_3T^{\frac{3}{8}}+c_4\tau^{\frac{1}{8}}
     \right)
     +\tau c_1\right)
\end{split}
\end{equation*}
for $s\in[0,T]$ and $\tau\ge0$.
Inequality one uses the exponential
decay Proposition~\ref{prop:semigroup-NEW},
the Lipschitz Lemma~\ref{le:f}
for $f$, and its Corollary~\ref{cor:f}.
To obtain line three we used~(\ref{eq:e-1}).
In line five we used backward time 
exponential decay~(\ref{eq:c-NEW}).
\\
To see inequality two observe the following.
Estimate the first integral in lines one and two
by~(\ref{eq:int-Gamma}), the second one
by~(\ref{eq:int-preGamma}).
Recall that $\gamma\in\Bb_{\rho_0}$ by our local setup.
Apply estimate~(\ref{eq:xi-xi}).
In addition, use~(\ref{eq:xi-xi}) to conclude that
$\norm{\Theta^T(s)}_X\le\rho_0 c_1$
whenever $s\in[0,T]$.
(Note that the same is true when $T$ is replaced by
$T+\tau$.)
The elements of $Z^T$ (and $Z^{T+\tau}$) take
values in $\Bb_{\rho_0}$ by step~1. 
Use that $c\kappa(\rho)\le1$
by~(\ref{eq:rho-backward}) and that
$e^{(s-T)\mu}\le1$.
To estimate the difference
$\xi^{T+\tau}(T+\tau)-\xi^T(T)\in X$ in line four is
surprisingly subtle. 
{\it This estimate will be carried
out separately below}; see~(\ref{eq:xi-xi-T-tau}) for
the result used in inequality two
and for the definition of $c_3$ and $c_4$.
\\
To obtain inequality three we used
smallness~(\ref{eq:rho_0-backward})
and~(\ref{eq:rho-backward}) of $\rho$ and
estimate~(\ref{eq:1-e}).
Now take the supremum over $s\in[0,T]$ to get
\begin{equation}\label{eq:V-V}
     \Norm{\Theta^{T+\tau}-\Theta^T}_{C^0([0,T],X)}
     \le c_T\rho_0 \tau^{\frac{1}{8}}
\end{equation}
where
$\frac12
     c_T
     =
     c_3 T^{\frac{3}{8}}
     +\tau^{\frac{7}{8}}
     \left(
       c_1^2
       +c^2\lambda_1^2+c\abs{\lambda_1}
       +c_1
     \right)
     +c_4 \tau^{\frac{1}{8}}
$.
Thus
\begin{equation*}
     \Norm{
     \tfrac{d}{dT}\Gg(T+\tau,\gamma,z_+)
     -\tfrac{d}{dT}\Gg(T,\gamma,z_+)
     }_X
     =\Norm{\Theta^{T+\tau}(0)-\Theta^T(0)}_X
     \le c_T\rho_0\tau^{\frac{1}{8}},
\end{equation*}
that is
$T\mapsto\frac{d}{dT}\Gg(T,\gamma,z_+)$ is locally H\"older
continuous with exponent $\alpha=\frac{1}{8}$.

{\it As mentioned above it remains to
estimate the $W^{1,2}$ norm of the difference:}
\begin{equation*}
\begin{split}
     \xi^{T+\tau}(T+\tau)-\xi^T(T)
    &=\left(e^{-\tau A}-\1\right) e^{-T A}z_+
  \\
    &\quad
     +\int_0^T e^{-(T+\tau-\sigma)A}\pi_+
     \left(
     f(\xi^{T+\tau}(\sigma))-f(\xi^T(\sigma))
     \right)
     \, d\sigma
  \\
    &\quad
     +\int_0^T
     \left(e^{-\tau A}-\1\right)
     e^{-(T-\sigma)A}\pi_+ f(\xi^T(\sigma))
     \, d\sigma
   \\
    &\quad
     +\int_T^{T+\tau}
     e^{-(T+\tau-\sigma)A}\pi_+
     f(\xi^{T+\tau}(\sigma))
     \, d\sigma.
\end{split}
\end{equation*}
We added zero to obtain terms II and III
in this sum I+II+III+IV of four.

I) Concerning term one we get
\begin{equation*}
\begin{split}
     \Norm{\left(e^{-\tau A}-\1\right) e^{-TA}z_+}_X
    &=\Norm{
     \int_0^\tau -A
     e^{-sA} e^{-TA}z_+ \, ds
     }_X
   \\
    &\le\int_0^\tau
     \Norm{e^{-sA}\pi_+}_{\Ll(X)}
     \Norm{Ae^{-TA}\pi_+}_{\Ll(X)}
     \Norm{z_+}_X
     \, ds
   \\
    &\le\int_0^\tau
     ce^{-s\mu} \left(\frac{c^\prime C}{T}
     e^{-T\mu}\right)
     \rho_0
     \, ds
   \\
    &\le\frac{cc^\prime C\rho_0e^{-T_0\mu}}{T_0}\;\tau.
\end{split}
\end{equation*}
The first identity even without norms
is standard; see
e.g.~\cite[Prop.~1.3.6.~(ii)]{LUNARDI-InetSem}.
To obtain inequality one we permuted $A$ and 
$e^{-sA}$; see e.g.~\cite[Thm.~1.3.3.~(i)]{LUNARDI-InetSem}.
Here we used that $e^{-TA}z_+\in W^{2,1}=D(A)$
since $T>0$.
Compare the above estimate on $X^+$
with the corresponding estimate~(\ref{eq:e-1}) on the
finite dimensional vector space $X^-$ and note how
boundedness of $A^-$ simplifies~(\ref{eq:e-1}).
Inequality two uses that the norms
$\norm{A\cdot}_{1,2}$ and $\norm{\cdot}_{3,2}$
are equivalent with constant $c^\prime$
by compactness of $S^1$ and $A$ being of second order.
The regularity-for-singularity
estimate~(\ref{eq:reg-for-sing-NEW}) with constant
$C=C(\mu)$ allows to get
from $W^{3,2}$ back to $W^{1,2}$ catching a factor
$CT^{-1}$. The final step uses~(\ref{eq:1-e}).

II)~For term two use estimate~(\ref{eq:xi-xi})
and the fact that $c\kappa(\rho)\le 1$
by~(\ref{eq:rho-backward}) to get that
\begin{equation*}
\begin{split}
     \int_0^T\Norm{ e^{-(T+\tau-\sigma)A}\pi_+
     \left(
     f(\xi^{T+\tau}(\sigma))-f(\xi^T(\sigma))
     \right)}_X
     d\sigma
    &\le\frac{\rho_0 c_1\tau}{\mu^\frac{1}{4}}
     \int_{\tau\mu}^{\tau\mu+T\mu}
     \frac{e^{-s}}{s^{\frac{3}{4}}}ds
   \\
    &\le\frac{\rho_0 c_1\tau}{\mu^\frac{1}{4}}\Gamma(\tfrac{1}{4})
   \\
    &\le\frac{4\rho_0 c_1}{\mu^\frac{1}{4}}\;\tau.
\end{split}
\end{equation*}
We also used the
definition of the $\Gamma$ function
after~(\ref{eq:int-preGamma})
and its functional equation.

III) Term three requires similar techniques as term one,
but their application requires more care.
Namely, it is crucial not to deal with the $\Ll(L^1,W^{3,2})$ norm
in one go, but to decompose it into a product involving
$\Ll(L^1,W^{1,q})$ and $\Ll(W^{1,q},W^{3,2})$ norms
where $q$ is any real strictly larger than the
order (two) of the differential operator $A$.
This way we avoid catching either a factor $s^{-\alpha}$ or
$T^{-\alpha}$ with $\alpha\ge 1$ when trading
regularity for singularity via~(\ref{eq:reg-for-sing-NEW}).
Each of these factors would void our estimate,
since they are not integrable locally
near zero. Pick $q>2$. 
Similarly as in case of term
one we obtain that
\begin{equation*}
\begin{split}
   &\int_0^T
    \Norm{\left(e^{-\tau A}-\1\right)
    e^{-(T-\sigma)A}\pi_+ f(\xi^T(\sigma))
    }_X
    \, d\sigma
   \\
    &\le\int_0^T\int_0^\tau
     \Norm{Ae^{-sA}e^{-(T-\sigma)A}\pi_+
     f(\xi^T(\sigma))
     }_{W^{1,2}}
     \, ds\, d\sigma
   \\
    &\le\int_0^T\int_0^\tau
     c^\prime\Norm{e^{-sA}e^{-(T-\sigma)A}\pi_+
     f(\xi^T(\sigma))
     }_{W^{3,2}}
     \, ds\, d\sigma
   \\
    &\le c^\prime\kappa(\rho)\rho_0
     \int_0^T\int_0^\tau
     \Norm{e^{-sA}\pi_+}_{\Ll(W^{1,q},W^{3,2})}
     \Norm{e^{-(T-\sigma)A}\pi_+}_{\Ll(L^1,W^{1,q})}
     \, ds\, d\sigma
   \\
    &\le c^\prime C^\prime C^{\prime\prime}\rho_0
     \int_0^\tau
     e^{-s\mu} s^{-\frac{3}{4}-\frac{1}{2q}}\, ds\,
     \int_0^T
     e^{-(T-\sigma)\mu}
     (T-\sigma)^{-\frac{1}{2}-\frac{1}{2q}}\, d\sigma
   \\
    &=c^\prime C^\prime C^{\prime\prime}\rho_0
     \int_0^\tau
     e^{-s\mu} s^{-\frac{7}{8}}\, ds\,
     \int_0^T
     e^{-s\mu} s^{-\frac{5}{8}}\, ds
   \\
    &\le
     22c^\prime C^\prime C^{\prime\prime}\rho_0
     T^{\frac{3}{8}} \tau^{\frac{1}{8}}.
\end{split}
\end{equation*}
Inequality three uses once more that
$\norm{\xi^T(\sigma))}_X\le\rho_0$ by step~1.
Note that $\kappa(\rho)\le 1$.
Inequality four uses twice
the regularity-for-singularity
estimate~(\ref{eq:reg-for-sing-NEW})
with constants
$C^\prime$ and $C^{\prime\prime}$, respectively.
The exponent
of $s$ shows that concerning integrability we could
have picked any $q>2$.
In the final inequality we dropped the factors
$e^{-s\mu}\le 1$ and carried out the integrals.

IV) Concerning term four  we get the estimate
\begin{equation*}
     \int_T^{T+\tau}\Norm{ e^{-(T+\tau-\sigma)A}\pi_+
     f(\xi^{T+\tau}(\sigma))}_X
     d\sigma
     \le\rho_0
     \int_0^{\tau} e^{-s\mu} s^{-\frac{3}{4}}
     \, ds
     \le 4\rho_0\tau^{\frac{1}{4}}
\end{equation*}
by dropping the term $e^{-s\mu}\le 1$ under the integral.

{\it Side remark concerning the
estimate for term III:}
Unfortunately, we do not see any way to trade
$\tau^{1/8}$ for $\tau$ or, equivalently,
to trade $T^{3/8}$ for $T$. This has the
following consequences.
The positive power of $T$ obstructs the conclusion
that $\frac{d}{dT}\Gg$ is uniformly continuous in $T$.
The conclusion of local Lipschitz continuity is
obstructed by the factor $\tau^\alpha$ with
$\alpha=1/8<1$.
All we can say is that $\frac{d}{dT}\Gg$ is locally
H\"older continuous in $T$ 
with exponent $\alpha=1/8$.

To summarize, the above estimates show that
\begin{equation}\label{eq:xi-xi-T-tau}
     \Norm{\xi^{T+\tau}(T+\tau)-\xi^T(T)}_X
     \le\tau^{\frac{1}{8}}\rho_0
     \left(c_3T^{\frac{3}{8}}
     +c_4\tau^{\frac{1}{8}}\right)
\end{equation}
for $\tau\ge0$ and where
$$
     c_3:=22c^\prime C^\prime C^{\prime\prime}
     ,\qquad
     c_4:=4+\tau^{\frac{7}{8}}\left( cc^\prime C T_0^{-1}
     +4c_1\mu^{-\frac{1}{4}}\right).
$$
This concludes the proof of~(\ref{eq:V-V})
and therefore of step~4.
\qed
\end{proof}

\noindent
{\bf Step 5.}
{\it For $T\ge0$ the graph map
$\Gg^T_\gamma:\Bb^+\to X^-\oplus X^+$,
$z_+\mapsto\left( G^T_\gamma(z_+),z_+\right)$, and its
inverse $\pi_+|_{\Gg^T_\gamma(\Bb^+)}$ are both Lipschitz
continuous with respect to the $W^{1,2}$ norm.
In fact, the graph map is a
diffeomorphism onto its image.
} 

\begin{proof}
For $j=1,2$ pick $z_j\in\Bb^+$
and denote the fixed point
$\xi^T_{\gamma,z_j}$ of 
$\Psi^T=\Psi^T_{\gamma,z_j}$ by $\xi_j$.
Similarly to the estimate in the 
proof of step~2 we obtain for each
$s\in[0,T]$ that
$$
     \Norm{\xi_1(s)-\xi_2(s)}_X
     \le ce^{-s\mu}
     \Norm{z_1-z_2}_X
     +\frac{1}{2}
     \Norm{\xi_1-\xi_2}_{\exp}.
$$
Multiply by $e^{s\frac{\mu}{2}}$
and take the supremum over 
$s\in[0,T]$ to get
\begin{equation}\label{eq:xi1-xi2}
     \Norm{\xi_1-\xi_2}_{\exp}
     \le 2 c 
     \Norm{z_1-z_2}_X.
\end{equation}
By~(\ref{eq:G=xi}) this proves
Lipschitz continuity of 
$\Gg^T_\gamma$, namely
\begin{equation*}
     \Norm{\Gg^T_\gamma(z_1)
      -\Gg^T_\gamma(z_2)}_X
     =\Norm{\xi_1(0)-\xi_2(0)}_X
     \le\Norm{\xi_1-\xi_2}_{\exp}
     \le 2c\Norm{z_1-z_2}_X.
\end{equation*}
Next use that $\pi_+$ vanishes on $X^-$ and acts as the
identity on $X^+$ to see that $\pi_+$
is a left inverse of $\Gg^T_\gamma$.
Thus $\pi_+$ restricted to $\Gg^T_\gamma(\Bb^+)$ is 
its inverse.
But this restriction is of class $C^1$, because it is
of the form $\pi_+\circ\Gg^T_\gamma(z_+)$
where $\pi_+$ is linear and the map
$z_+\mapsto\Gg^T_\gamma(z_+)
:=\left(G^T_\gamma z_+, z_+\right)$ is
of class $C^1$ by step~3.

To see that the restriction of $\pi_+$ to
$\Gg^T_\gamma(\Bb^+)$ is Lipschitz continuous
consider the difference  
$\xi_1(0)-\xi_2(0)=\Psi^T\xi_1(0)-\Psi^T\xi_2(0)$ whose
right hand side is given by~(\ref{eq:Psi}). Apply
$\norm{a-b}\ge\norm{a}-\norm{b}$ with $a=z_1-z_2$
and~(\ref{eq:est-step-2}) for $s=0$
to get
\begin{equation*}
     \Norm{\xi_1(0)-\xi_2(0)}_X
     \ge \Norm{z_1-z_2}_X
     - c\kappa(\rho)\frac{2}{3\mu}
     \Norm{\xi_1-\xi_2}_{\exp}.
\end{equation*} 
By~(\ref{eq:xi1-xi2}) and the smallness
assumption~(\ref{eq:rho-backward}) on $\rho$ this
implies that
$$
     \Norm{\xi_1(0)-\xi_2(0)}_X
     \ge\left(1-c\kappa(\rho)
     \frac{2}{3\mu} 2c\right)
     \Norm{z_1-z_2}_X
     \ge\frac{1}{2}\Norm{z_1-z_2}_X
$$
which by~(\ref{eq:G=xi}) and the fact that $\pi_+$ left
inverts $\Gg^T_\gamma$ is equivalent to
\begin{equation}\label{eq:bi-Lip2}
     \Norm{\Gg^T_\gamma(z_1)
     -\Gg^T_\gamma(z_2)}_X
     \ge\frac{1}{2}
     \Norm{\pi_+\Gg^T_\gamma(z_1)
     -\pi_+\Gg^T_\gamma(z_2)}_X.
\end{equation}
This proves that $\pi_+$ is Lipschitz continuous on the
image of $\Gg^T_\gamma$.

By estimate~(\ref{eq:bi-Lip2}) and the estimate 
after~(\ref{eq:xi1-xi2}) the map $\Gg^T_\gamma$ is
bi-Lipschitz and therefore a homeomorphism onto
its image. Since the map and its inverse are both of
class $C^1$, it is in fact a diffeomorphism onto its image.
\qed
\end{proof}

\noindent
{\bf Step 6.} (Uniform convergence)
{\it 
$
     \norm{\Gg^\infty(z_+)
     -\Gg_\gamma^T(z_+)}_{W^{2,2}}
     \le \rho_0
     e^{-T\frac{\mu}{16}}
$
$\forall T\ge T_2$.
}

\begin{proof}
Assume $T\ge T_2$; see~(\ref{eq:back-traj}) below.
Consider the fixed point $\xi^T=\xi^T_{\gamma,z_+}$ of
$\Psi_{\gamma,z_+}^T$ on $Z^T$ and the fixed point
$\eta=\eta_{z_+}$ of $\Psi_{z_+}$
on $Z$ defined by~(\ref{eq:exp norm}).
Because $\Gg^T_\gamma(z_+)=\xi^T(0)$ by~(\ref{eq:G=xi}),
similarly $\Gg^\infty(z_+)=\eta(0)$,
it remains to estimate the difference
$\eta(0)-\xi^T(0)$.
Observe that, firstly, since the difference lies in
$X^-\subset C^\infty$ application of the $W^{2,2}$
norm makes sense. Secondly, by the respective
representation formulae, this difference depends on
the whole trajectories $\eta$ and
$\xi^T$. 
But while $\eta$ runs into the origin, the
trajectory $\xi^T$ ends on the fiber $\Dd_\gamma$
far away! So the difference
$\eta-\xi^T$ cannot converge to zero, as
$T\to\infty$, uniformly on $[0,T]$. However,
Figure~\ref{fig:fig-xi-xi} suggests that this
could be true on some initial part of the domain
$[0,T]$, say on $[0,\frac{1}{2}T]$. So step~A
is to reduce the problem to the smaller interval
$[0,\frac{1}{2}T]$. Step~B is to solve the
reduced problem. Here the key idea is to
\emph{suitably partition} both trajectories $\eta$ and
$\xi^T$ and compare due parts; see
Figure~\ref{fig:fig-exp-convergence}.
The fact that $\eta$ is
\emph{asymptotically well behaved}, i.e. exponentially
close to zero on $[T,\infty)$, enters frequently.
\begin{figure}
  \includegraphics{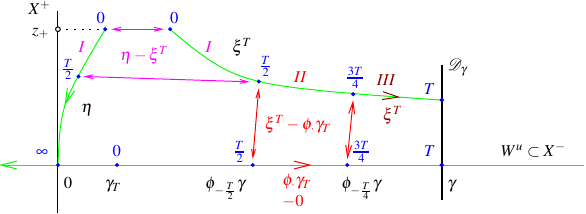}
  \caption{Time partitions and exponentially
           decaying differences}
  \label{fig:fig-exp-convergence}
\end{figure}

We proceed as follows:
 In step~A we estimate the (stronger)
$W^{2,2}$ norm of the difference $\eta(0)-\xi^T(0)$ by an
exponentially decaying function of $T$ plus the
supremum over $s\in[0,\frac{1}{2}T]$ of the (weaker) norm
$\norm{\eta(s)-\xi^T(s)}_{X}$. This reduction of a stronger
to a weaker norm is based on the \emph{key fact} that the
difference $\eta(0)-\xi^T(0)$ only involves $\pi_-$ terms.
Namely, these take values in $X^-$, hence in $C^\infty$.
In step~B we prove exponential decay of this sup norm.
Here we encounter again
the difference $\eta-\xi^T$, unfortunately on the whole
interval $[0,T]$. Now the \emph{key idea} is to decompose
this interval into three pieces, namely
$$
     I:=[0,\tfrac{1}{2}T]
     ,\qquad
     II:=[\tfrac{1}{2}T,\tfrac{3}{4} T]
     ,\qquad
     III:=[\tfrac{3}{4} T,T],
$$
as shown in Figure~\ref{fig:fig-exp-convergence}.
In fact an extra piece $[T,\infty)$ is brought in
by $\eta$.
On interval $I$ we pull out the supremum norm and use
smallness of the Lipschitz constant $\kappa(\rho)$ to
get a coefficient less than one to throw the whole
$(\eta-\xi^T)$ term on the left hand side. Off $I$
we apply the triangle inequality to deal with each term
$\eta$ and $\xi^T$ separately. Exponential decay
built into the definition~(\ref{eq:exp norm}) of $Z$
allows to handle $\eta$ on its whole remaining time
interval $[\frac{1}{2}T,\infty)$ in one go. It remains
to deal with $\xi^T$ on intervals $II$ and $III$. For
$\sigma\in II$ we exploit
(after adding zero) that both terms
$\xi^T(\sigma)-\phi_\sigma(\gamma_T)$ and
$\phi_\sigma(\gamma_T)$ individually decay
exponentially in $T$, uniformly in $\sigma\in II$.
For the first term this is simply true by
definition~(\ref{eq:ZT}) of $Z^T$. Concerning the
second term we use that $\gamma$ lies in the unstable
manifold. Hence $\phi_\sigma(\gamma_T)=\phi_t(\gamma)$
collapses exponentially fast into the origin, since
$t:=\sigma-T\in[-\frac{1}{2}T,-\frac{1}{4} T]$ and
the whole interval sets off to $-\infty$~\footnote{
  The argument relies on the right boundary
  of the $t$-interval running to $-\infty$, as
  $T\to\infty$. Therefore the right boundary of $II$
  needs to be strictly smaller than $T$, but at the
  same time be element of $[0,T]$ whichever $T$ we
  pick. Thus any $\alpha T$ with $0<\alpha<1$ is a
  good choice.
};
cf. Remark~\ref{rem:unstable-manifold}.
For interval $III$ the argument is \emph{analytic} and
\emph{cannot be guessed}
by Figure~\ref{fig:fig-exp-convergence}. The figure
even suggests trouble. Fortunately, we are not
concerned with the image of the trajectory, but
with the integral over its time parametrization.
In fact due to an abundance of negative powers
already the coarse estimate
$\norm{\xi^T(\sigma)}_X\le\rho_0$ is fine:
It leaves us with integrating
$e^{(s-\sigma)\mu}$ over $III$. But
$s\le\frac{1}{2} T$ by assumption and
$\sigma\ge\frac{3}{4} T$ on $III$.~\footnote{
  Exponential decay is achieved, if
  the left boundary of $III$ is of the
  form $\alpha T$ with $\alpha>\frac{1}{2}$.
}

Our choice of time partitions and
combinations of trajectory pieces which leads to
exponential decay in $T$ is
shown in Figure~\ref{fig:fig-exp-convergence}
where the upper labels of points are time.
It is instructive to figure out how the drawing changes
as $T$ tends to infinity. How do $\eta$ and $\xi^T$
change and how their time labels?
What happens to the lengths of the four double arrows?
Consider the pair of double arrows with
common point $\xi^T(T/2)$. What is the
asymptotic behavior of this point?

(A)~Abbreviate ${\tilde{X}}:=W^{2,2}$.
Note that by parabolic regularity
the heat flow trajectories $\eta$ and $\xi^T$ take values
in $C^\infty$ at strictly positive times. Recall that
$\gamma\in S^u_\eps\subset(X^-\cap \Bb_{\rho_0})$
by our local setup.
Use formula~(\ref{eq:Psi}) for $\xi^T$
and the one for $\eta$, see formula after~(\ref{eq:exp norm}),
together with the fact that the nonlinearity $f$ maps $X$
to $Y$ to obtain~\footnote{
  Here and throughout
  $\left(\int_a^b +\int_c^d\right) f$
  abbreviates $\int_a^b f+\int_c^d f$.
  }
\begin{equation*}
\begin{split}
    &\Norm{\eta(0)-\xi^T(0)}_{\tilde X}
   \\
     &\le
      \left(\int_0^{T/2}
      +\int_{T/2}^T\right)
      \Norm{e^{\sigma A^-}\pi_-}
      _{\Ll(Y,\tilde X)}
      \Norm{f\circ\eta(\sigma)-f\circ\xi^T(\sigma)}_Y 
      \,d\sigma
   \\
     &\quad
      +\Norm{e^{TA^-}\pi_-}
      _{\Ll(X,\tilde X)}\Norm{\gamma}_X
      +\int_T^\infty
      \Norm{e^{\sigma A^-}\pi_-}
      _{\Ll(Y,\tilde X)}
      \Norm{f\circ\eta(\sigma)}_Y 
      \, d\sigma
    \\
     &\le
      c\kappa(\rho)
      \Norm{\eta-\xi^T}_{C^0([0,\frac{T}{2}],X)}
      \int_0^{T/2}
      e^{-\sigma\mu}
      \, d\sigma
      +c\kappa(\rho)\cdot 2\rho_0
      \int_{T/2}^T e^{-\sigma\mu}
      \, d\sigma
   \\
     &\quad
      +c\rho_0 e^{-T\mu}
      +c\kappa(\rho)
      \rho_0
      \int_T^\infty
      e^{-\sigma\mu}
      \,d\sigma.
   \\
    &\le
     \frac{c\kappa(\rho)}{\mu}
     \Norm{\eta-\xi^T}_{C^0([0,\frac{T}{2}],X)}
     +
     \frac{2c\kappa(\rho)}{\mu}
     \rho_0 e^{-T\frac{\mu}{2}}
     +
     \frac{\rho_0}{8} e^{-T\frac{\mu}{2}}
     +
     \frac{c\kappa(\rho)}{\mu}
     \rho e^{-T\mu}
   \\
    &\le
     \frac{1}{8}
     \Norm{\eta-\xi^T}_{C^0([0,\frac{T}{2}],X)}
     +
     \frac12 \rho_0 e^{-T\frac{\mu}{2}}.
\end{split}
\end{equation*}
Inequality two uses the exponential decay
Proposition~\ref{prop:semigroup-NEW}~(c) and
the Lipschitz Lemma~\ref{le:f} for $f$ and $p=1$.
We also used definition~(\ref{eq:exp-T})
of the exp-$T$ norm and the fact that the
elements of $Z^T$ take values in
$\Bb_{\rho_0}$ by step~1 and those of $Z$ in
$\Bb_\rho\subset\Bb_{\rho_0}$ by
definition~(\ref{eq:exp norm}).
Inequalities three is by calculation
and definition of $T_2$.
Now use~(\ref{eq:rho-backward}).

(B)~Pick $s\in[0,\frac{T}{2}]$.
Similarly as in~(A) we get the estimate
\begin{equation*}
\begin{split}
    &\Norm{\eta(s)-\xi^T(s)}_X
   \\
     &\le
      \int_0^s\Norm{e^{-(s-\sigma)A}
      \pi_+}_{\Ll(Y,X)}
      \Norm{f\circ\eta(\sigma)-f\circ \xi^T(\sigma)}_Y 
      \, d\sigma
   \\
     &\quad
      +
      \left(
      \int_s^{\frac{T}{2}}
      +\int_{\frac{T}{2}}^{\frac{3T}{4}}
      +\int_{\frac{3T}{4}}^T
      \right)
      \Norm{e^{-(s-\sigma)A^-}\pi_-}_{\Ll(Y,X)}
      \Norm{f\circ\eta(\sigma)-f\circ\xi^T(\sigma)}_Y
      d\sigma
   \\
     &\quad
      +
      ce^{(s-T)\mu}\Norm{\gamma}_X
      +
      \int_T^\infty
      \Norm{e^{-(s-\sigma)A^-}\pi_-}
      _{\Ll(Y,X)}
      \Norm{f\circ\eta(\sigma)}_Y 
      d\sigma
    \\
\end{split}
\end{equation*}
\begin{equation*}
\begin{split}
    &\le
     c\rho_0 e^{-T\frac{\mu}{2}}
     +
     c\kappa(\rho)
     \Norm{\eta-\xi^T}
      _{C^0([0,\frac{T}{2}],X)}
     \left(
     \int_0^s
     \frac{e^{-(s-\sigma)\mu}}
     {(s-\sigma)^{\frac{3}{4}}}
     \, d\sigma
     +
     \int_s^{\frac{T}{2}}
     e^{(s-\sigma)\mu} d\sigma
     \right)
   \\
     &\quad
      +
      c\kappa(\rho)
      \int_{\frac{T}{2}}^{\frac{3T}{4}}
      e^{(s-\sigma)\mu}
      \left(
      \Norm{\xi^T(\sigma)
      -\phi_\sigma (\gamma_T)}_X
      +
      \Norm{\phi_\sigma(\gamma_T)}_X
      \right)
      d\sigma
   \\
     &\quad
      +
      c\kappa(\rho)
      \int_{\frac{3T}{4}}^T
      e^{(s-\sigma)\mu}
      \Norm{\xi^T(\sigma)}_X
      d\sigma
      +
      c\kappa(\rho)
      \Norm{\eta}_{\exp}
      \int_{\frac{T}{2}}^\infty
      e^{(s-\frac{3}{2}\sigma)\mu}
      d\sigma.
\end{split}
\end{equation*}
To get overall exponential decay in $T$ we have
split the domain of integration in three parts.
The domain of integration $\int_{T/2}^\infty$ in the last line
is not a misprint.

To continue the estimate consider the last three lines.
Now we explain how to
get to the corresponding three lines 
in~(\ref{eq:arrive-at}) below. Concerning line one
use the definition of $T_2$ and
recall that $s\in[0,T/2]$ and
use~(\ref{eq:int-preGamma}) and~(\ref{eq:int-Gamma}).
In line two we drop $e^{(s-\sigma)\mu}\le1$
and use that
$
     \norm{\xi^T(\sigma)
     -\phi_\sigma(\gamma_T)}_X
     \le\rho e^{-\sigma\frac{\mu}{2}}
$
by definition of $Z^T$ and that
\begin{equation}\label{eq:back-traj}
     \int_{\frac{T}{2}}^{\frac{3T}{4}}
     \Norm{\phi_{\sigma-T}\gamma}_X d\sigma
     =
     \int_{\frac{T}{8}}^{\frac{3T}{8}}
     \Norm{\phi_{-t-\frac{T}{8}}\gamma}_X  dt
     \le
     \int_{\frac{T}{8}}^{\frac{3T}{8}}
     \rho e^{-t\frac{\mu}{2}}\, dt
     \le
     \frac{2\rho}{\mu} e^{-T\frac{\mu}{16}}.
\end{equation}
Here the identity is by change of variables
$t=-\sigma+\frac{7}{8}T$ and the first inequality
uses Remark~\ref{rem:unstable-manifold} for the
backward time trajectory
$\tilde\eta(-t)=\phi_{-t-T/8}(\gamma)$ defined for $t\ge0$.
To see this note that $\tilde\eta(-t)\to 0$, as $t\to\infty$,
because $\gamma$ lies in the descending sphere
$S^u_\eps$ by assumption.
Observe that the image of $\tilde\eta$ is contained
in the backward flow invariant set
$\phi_{-T/8}\overline{W^u_\eps}$ which
by assumption on $T_2$ is itself contained
in $\Bb^-:=\Bb_{\rho/2c}\cap X^-\subset\Bb_\rho$.
By the argument in
Remark~\ref{rem:unstable-manifold}
the solution
$\tilde\eta$ is equal to the unique fixed point of the map
$\Phi_\gamma$.
In particular, it holds that $\tilde\eta\in Z^u$
and therefore
$\norm{\tilde\eta(-t)}_X\le\rho e^{-t\mu/2}$
for every $t\ge0$.
To summarize, line two is bounded from above by
\begin{equation*}
      c\kappa(\rho)\cdot\rho
      \left(
     \int_{\frac{T}{2}}^{\frac{3T}{4}}
     e^{-\sigma\frac{\mu}{2}}
     \, d\sigma
     +\frac{2e^{-T\frac{\mu}{16}}}{\mu}
     \right)
     \le
     c\kappa(\rho)\cdot\frac{2\rho}{\mu}
     \left(
     e^{-T\frac{\mu}{4}}
     +e^{-T\frac{\mu}{16}}
     \right).
\end{equation*}

In line three use
$\norm{\xi^T(\sigma)}_X\le\rho_0$ for any
$\xi^T\in Z^T$ by step~1 and
$\norm{\eta}_{\exp}\le\rho$ by definition of $Z$.
Carry out the integrals, in the second one
drop $e^{(s-\sigma)\mu}\le1$, to get
\begin{equation}\label{eq:arrive-at}
\begin{split}
     \Norm{\eta(s)-\xi^T(s)}_X
     &\le 
      \frac{\rho_0}{8} e^{-T\frac{\mu}{4}}
      +
      c\kappa(\rho)
      \left(\frac{8}{\mu^{1/4}}
      +\frac{1}{\mu}\right)
      \Norm{\eta-\xi^T}
      _{C^0([0,\frac{T}{2}],X)}
   \\
      &\quad
      +
      \frac{4c\kappa(\rho)}{\mu}\rho
      e^{-T\frac{\mu}{16}}
      +
      \frac{c\kappa(\rho)}{\mu}\rho_0
      e^{-T\frac{\mu}{4}}
      + 
      \frac{2c\kappa(\rho)}{\mu}
      \rho
      e^{-T\frac{\mu}{4}}
   \\
    &\le
     \frac{1}{8}
     \Norm{\eta-\xi^T}
      _{C^0([0,\frac{T}{2}],X)}
     +
     \frac{1}{2}
     \rho_0 e^{-T\frac{\mu}{16}}.
\end{split}
\end{equation}
The last step uses
smallness~(\ref{eq:rho-backward}) of $\rho$. Take the
sup over $s\in [0,\frac{T}{2}]$ to get
\begin{equation}\label{eq:xi-eta}
     \Norm{\eta-\xi^T}
     _{C^0([0,\frac{T}{2}],X)}
     \le\frac{4}{7} \rho_0 e^{-T\frac{\mu}{16}}.
\end{equation}
Hence
$
     \norm{\Gg^\infty(z_+)-\Gg^T_\gamma(z_+)}_{\tilde X}
     =\norm{\eta(0)-\xi^T(0)}_{\tilde X}
     \le \rho_0
     e^{-T\frac{\mu}{16}}
$,
for all $\gamma\in S^u_\eps$, times $T\ge T_2$, and
$z_+\in\Bb^+$ and this proves step~6.
\qed\end{proof}
The Sobolev embedding
$W^{2,2}(S^1)\hookrightarrow C^1(S^1)$
concludes
the proof of
Theorem~\ref{thm:backward-lambda-Lemma}.

\subsection{Proof of uniform $C^1$ convergence (Theorem~\ref{thm:uniform-C1})}

Theorem~\ref{thm:uniform-C1}
builds on the backward $\lambda$-Lemma,
Theorem~\ref{thm:backward-lambda-Lemma}.
So we may use any of the six steps of its proof.
The proof at hand takes two steps.
Fix $\gamma\in S^u_\eps$ and $z_+\in\Bb^+$.

\vspace{.1cm}
\noindent
{\bf Step I.} {\rm ($L^2$ extension)}
{\it 
$
     \norm{d\Gg^T_\gamma(z_+)
     v}_2
     \le 2 \norm{v}_2
$
for all $v\in \pi_+(L^2)$ and 
$T\ge T_1$.
}

\begin{proof}
By the bounded linear transform
theorem~\cite[Thm.~I.7]{ReSi-I}
it suffices to pick $v$ in the dense subspace 
$X^+=\pi_+(X)$ of $\pi_+(L^2)$. 
Pick $\tau\ge0$ small.
Consider the fixed point
$\xi_{z_++\tau v}=\xi^T_{\gamma,z_++\tau v}\in Z^T$ of
$\Psi^T_{\gamma,z_++\tau v}$.
By~(\ref{eq:Psi}) the fixed point property means that
\begin{equation}\label{eq:xiT-tau}
\begin{split}
     \xi_{z_++\tau v}(s)
    &=
     e^{-sA}\left(z_++\tau v\right)
     +\int_0^s e^{-(s-\sigma)A}\pi_+
     f(\xi_{z_++\tau v}(\sigma))\, 
     d\sigma\\
    &\quad
     +e^{-(s-T)A^-}\gamma
     -\int_s^T e^{-(s-\sigma)A^-}
     \pi_-
     f(\xi_{z_++\tau v}(\sigma))\,
     d\sigma
\end{split}
\end{equation}
for every $s\in[0,T]$.
By the proof of step~3 the
composition of maps
$
     \tau
     \mapsto
     \xi_{z_++\tau v}
     \mapsto
     \xi_{z_++\tau v}(s)
$
is of class $C^1$.
Hence the linearization
is well defined and satisfies
\begin{equation}\label{eq:X_v}
\begin{split}
     X_v(s)
   :&=\left.\tfrac{d}{d\tau}\right|
     _{\tau=0}
     \xi_{z_++\tau v}(s)
  \\
    &=e^{-sA} v
     +\int_0^s e^{-(s-\sigma)A}\pi_+
     \left(
     df|_{\xi_{z_+}(\sigma)}\circ
     X_v(\sigma)\right)
     \, d\sigma\\
    &\qquad\quad\;\;
      -\int_s^T e^{-(s-\sigma)A^-}\pi_-
     \left(
     df|_{\xi_{z_+}(\sigma)}\circ
     X_v(\sigma)
     \right)
     \, d\sigma
\end{split}
\end{equation}
for each $s\in[0,T]$.
Use~(\ref{eq:G=xi}) to see that
$
     X_v(0)
     =\left.\frac{d}{d\tau}\right|
     _{\tau=0}
     \xi_{z_++\tau v}(0)
     =d\Gg^T_\gamma(z_+) v
$.
To conclude the proof
it remains to show that
$\norm{X_v(0)}_2\le 2\norm{v}_2$.
Recall the estimate
\begin{equation}\label{eq:L^1-L^2}
     \Norm{e^{-sA}\pi_+}
     _{\Ll(L^2,X)}
     \le c s^{-\frac{1}{2}}e^{-s\mu}, \qquad s>0,
\end{equation}
provided by
Proposition~\ref{prop:semigroup-NEW}.
This motivates, cf.~\cite{Henry-81-GeomTheory},
to define the weighted exp norm
\begin{equation*}
     \Norm{X_v}_{\frac{1}{2},\exp}=
     \Norm{X_v}_{\frac{1}{2},\exp,T}
     :=\sup_{s\in[0,T]}
     s^{\frac{1}{2}} e^{s\frac{\mu}{2}}
     \Norm{X_v(s)}_X.
\end{equation*}
This choice allows to estimate
$\norm{X_v(s)}_X$ (up to a singular factor)
in terms of $\norm{v}_2$ instead of $\norm{v}_X$.
Namely, by~(\ref{eq:X_v}) and since
$v\in X^+\subset X\hookrightarrow L^2$ we obtain
that
\begin{equation*}
\begin{split}
     s^{\frac{1}{2}}e^{s\frac{\mu}{2}}
     \Norm{X_v(s)}_X
    &\le
     s^{\frac{1}{2}}e^{s\frac{\mu}{2}}
     \Norm{e^{-sA}\pi_+}
     _{\Ll(L^2,X)}
     \Norm{v}_2
     \\
    &\quad +
     s^{\frac{1}{2}}e^{s\frac{\mu}{2}}
     \int_0^s
     \Norm{e^{-(s-\sigma)A}\pi_+}
     _{\Ll(Y,X)} \kappa(\rho) 
     \Norm{X_v(\sigma)}_X
     \, d\sigma
     \\
    &\quad +
     s^{\frac{1}{2}}e^{s\frac{\mu}{2}}
     \int_s^T
     \Norm{e^{-(s-\sigma)A^-}\pi_-}
     _{\Ll(Y,X)} 
     \kappa(\rho) 
     \Norm{X_v(\sigma)}_X
     d\sigma
     \\
    &\le
     ce^{-s\frac{\mu}{2}}\Norm{v}_2
     + c\kappa(\rho)
     \Norm{X_v}_{\frac{1}{2},\exp}
     \int_0^s
     \frac{e^{-(s-\sigma)\frac{\mu}{2}}}
     {(s-\sigma)^{\frac{3}{4}}}
     \left(\frac{s}{\sigma}\right)
     ^{\frac{1}{2}}
     d\sigma
     \\
    &\quad +
     c\kappa(\rho)
     \Norm{X_v}_{\frac{1}{2},\exp}
     \int_s^T
     e^{\frac{3}{2}(s-\sigma)\mu}
     \left(\frac{s}{\sigma}\right)
     ^{\frac{1}{2}}
     d\sigma
     \\
    &\le
     ce^{-s\frac{\mu}{2}}\Norm{v}_2
     +c\kappa(\rho)
     \left(
     \frac{18}{\mu^{1/4}}
     +\frac{2}{3\mu}\right)
     \Norm{X_v}_{\frac{1}{2},\exp}
\end{split}
\end{equation*}
for every $s\in[0,T]$.
Inequality one uses that $\xi_{z_+}\in Z^T$
takes values in $B_{\rho_0}\subset\Uu$
by step~1. Hence Corollary~\ref{cor:f}
applies and provides the estimate for $df$.
In inequality two
we used that 
$\norm{X_v(\sigma)}_X
\le\sigma^{-\frac{1}{2}} e^{-\sigma\frac{\mu}{2}}
\norm{X_v}_{\frac{1}{2},\exp}$
by definition of the exp norm. 
We used~(\ref{eq:L^1-L^2}) to
obtain the first term
and Proposition~\ref{prop:semigroup-NEW}
to obtain the other two terms
of the sum. Inequality three will be proved below.
Now use smallness~(\ref{eq:rho-backward})
of $\rho$ and take the supremum
over $s\in[0,T]$ to obtain
\begin{equation}\label{eq:X_v-exp}
     \Norm{X_v}_{\frac{1}{2},\exp}
     \le 2c \Norm{v}_2.
\end{equation}

Concerning inequality three we need to estimate
the two integrals. Observe first of all that
$
     \int_s^T
     e^{\frac{3}{2}(s-\sigma)\mu}
     \left(\frac{s}{\sigma}\right)
     ^\frac{1}{2}
     \, d\sigma
     \le
     \int_s^T
     e^{\frac{3}{2}(s-\sigma)\mu}
     \, d\sigma
     \le
     \frac{2}{3\mu}
$
and
\begin{equation}\label{eq:3/4}
\begin{split}
    &\int_0^s
     \frac{e^{-(s-\sigma)\frac{\mu}{2}}}
     {(s-\sigma)^{\frac{3}{4}}}
     \left(\frac{s}{\sigma}
     \right)^{\frac{1}{2}}
     d\sigma
     \\
    &=
     \int_0^{s/2}
     \underbrace{
     e^{-(s-\sigma)\frac{\mu}{2}}}
       _{\le e^{-s\mu/4}}
     (\underbrace{s-\sigma}
       _{\ge s/2})^{-\frac{3}{4}}
     s^{\frac{1}{2}}
     \sigma^{-\frac{1}{2}}
     d\sigma
     +
     \int_{s/2}^s
     \frac{e^{-(s-\sigma)\frac{\mu}{2}}}
     {(s-\sigma)^{\frac{3}{4}}}
     \,(\, 
     \underbrace{s/\sigma}
       _{\le 2}
     \,)^{\frac{1}{2}}
     d\sigma
     \\
    &\le
     2^{\frac{3}{4}} s^{-\frac{1}{4}}
     e^{-s\frac{\mu}{4}}
     \int_0^{s/2}
     \sigma^{-\frac{1}{2}} d\sigma
     +
     \frac{2^{\frac{3}{4}}}{\mu^{\frac{1}{4}}}
     \Gamma(\tfrac{1}{4})
   \\
    &\le
     \frac{2}{\mu^{\frac{1}{4}}}
     +
     \frac{8}{\mu^{\frac{1}{4}}}.
\end{split}
\end{equation} 
Here we used that the last integral is equal to
$\sqrt{2s}$ and $h(s)
:=2^{\frac{5}{4}}s^{\frac{1}{4}}
e^{-s\frac{\mu}{4}}$
is bounded by $h(s_{max})=h(1/\mu)
=2(2/\mu e)^{1/4}$. 
Furthermore, we used~(\ref{eq:int-Gamma}).

We start over estimating $X_v(s)$, but now
at $s=0$ and in the $L^2$ norm.
Similarly as above, using that
$\norm{X_v(\sigma)}_X
\le 2c\sigma^{-\frac{1}{2}} 
e^{-\sigma\frac{\mu}{2}}\norm{v}_2$
by~(\ref{eq:X_v-exp}) we get
\begin{equation*}
\begin{split}
     \Norm{X_v(0)}_2
    &\le
     \Norm{v}_2
     +
     \int_0^T
     \Norm{e^{\sigma A^-}\pi_-}
     _{\Ll(L^1,L^2)} 
     \kappa(\rho) 
     \Norm{X_v(\sigma)}_X
     d\sigma
     \\
    &\le
     \Norm{v}_2
     +2c^2\kappa(\rho)
     \Norm{v}_2
     \int_0^T
     e^{-\frac{3}{2}\sigma\mu}
     \sigma^{-\frac{1}{2}}
     \, d\sigma
     \\
    &\le
     \Norm{v}_2
     +c^2\kappa(\rho)
     \left(\frac{6}{\mu^{1/4}}
     +\frac{3}{\mu^{5/4}}\right)
     \Norm{v}_2
     \le2\Norm{v}_2
\end{split}
\end{equation*} 
for $s\in[0,T]$.
Inequality one also
uses that $e^{-sA}$ restricts to a strongly
continuous semigroup on $L^2$ by
Proposition~\ref{prop:semigroup-NEW}
and that 
$\norm{\cdot}_{\Ll(L^1,L^2)}
\le\norm{\cdot}_{\Ll(L^1,X)}$
by the embedding 
$X\hookrightarrow L^2$.
Inequality four is by smallness~(\ref{eq:rho-backward})
of $\rho$. Concerning inequality three
we applied (for $s=0$)
the following consequence of H\"older's
inequality on the domain $[s,\infty)$, namely
\begin{equation}\label{eq:Hoelder}
     \int_s^\infty e^{-\frac{3}{2}\sigma\mu}
     \sigma^{-\frac{1}{2}} d\sigma
     \le
     \norm{e^{-\sigma\mu}
     }_{L^4}
     \norm{e^{-\sigma\frac{\mu}{2}}
     \sigma^{-\frac{1}{2}}
     }_{L^{4/3}}
     \le
     \left(3+\frac{3}{2\mu}\right)
     \frac{e^{-s\mu}}{\mu^{\frac{1}{4}}}
\end{equation}
for $s\ge 0$.
Here step two uses that
$\norm{e^{-\sigma\mu}}_{L^4}
=(1/4\mu)^{1/4} e^{-s\mu}$
by calculation and that
$$
     \norm{e^{-\sigma\frac{\mu}{2}}
     \sigma^{-\frac{1}{2}}}_{L^{\frac{4}{3}}}^{\frac{4}{3}}
     =
     \int_s^\infty
     e^{-\sigma\frac{2}{3}\mu}
     \sigma^{-\frac{2}{3}}
     \, d\sigma
     \le 
     \int_0^1\sigma^{-\frac{2}{3}}
     \, d\sigma
     +
     \int_1^\infty 
     e^{-\sigma\frac{2}{3}\mu}
     \, d\sigma
     =
     3+\frac{3}{2\mu} 
     e^{-\frac{2}{3}\mu}.
$$
This proves Step~I.
\qed\end{proof}

\noindent
{\bf Step II.} 
{\it 
$
     \norm{d\Gg^T_\gamma(z_+)v
     -d\Gg^\infty(z_+)v}_2
     \le e^{-T\frac{\mu}{16}}
     \norm{v}_2
$
$\;\forall T\ge T_0$
$\;\forall v\in \pi_+(L^2)$.
}

\begin{proof}
The proof of convergence of the linearized graph maps
should use convergence of the graph maps themselves.
Indeed~(\ref{eq:xi-eta}) is a key ingredient.
Another one is the Lipschitz estimate
for $df$ provided by Lemma~\ref{le:f}.

Pick $T\ge T_0$ and $v\in X^+$. Consider the fixed point
$\xi_{z_+}=\xi^T_{\gamma,z_+}$ of the strict contraction
$\Psi^T_{\gamma,z_+}$ on $Z^T$ and the fixed point
$\eta_{z_+}$ of $\Psi_{z_+}$ on $Z$. It is a side remark that
Theorem~\ref{thm:local-stable-manifold-graph}
is recovered by the present setup for $T=\infty$ and $\gamma:=0$.
For $\tau\ge0$ small $\xi_{z_++\tau v}$
satisfies the integral equation~(\ref{eq:xiT-tau})
and $\eta_{z_++\tau v}$ satisfies~(\ref{eq:xiT-tau})
with $T=\infty$; in particular, term three in that sum
disappears. Consider the linearizations
$
     X_v
     :=\left.\frac{d}{d\tau}
     \right|_{\tau=0}
     \xi^T_{\gamma,z_++\tau v}
$
and
$
     Y_v
     :=\left.\frac{d}{d\tau}
     \right|_{\tau=0}
     \eta_{z_++\tau v}
$.
Observe that $X_v$
satisfies the integral 
equation~(\ref{eq:X_v})
and $Y_v$ satisfies~(\ref{eq:X_v})
with $T=\infty$.
We know that $d\Gg^T_\gamma(z_+) v= X_v(0)$
by the identity following~(\ref{eq:X_v}),
similarly $d\Gg^\infty(z_+) v=Y_v(0)$.
It remains to estimate
$\norm{X_v(0)-Y_v(0)}_2$.
Define
$$
     \Norm{X_v}_*
     :=\sup_{s\in[0,\frac{1}{2}T]}
     s^{\frac{1}{2}}
     \Norm{X_v(s)}_X
$$
and abbreviate $\xi:=\xi^T_{\gamma,z_+}$
and $\eta:=\eta_{z_+}$.
Then we obtain the $L^2$ estimate
\begin{equation*}
\begin{split}
    &\Norm{X_v(0)-Y_v(0)}_2
   \\
    &\le \left(\int_0^{{\frac{T}{2}}}
     +\int_{{\frac{T}{2}}}^T\right)
     \Norm{e^{\sigma A^-}\pi_-}
     _{\Ll(L^1,L^2)}
     \Norm{
     df|_{\xi(\sigma)}
     \circ X_v(\sigma)
     -df|_{\eta(\sigma)}
     \circ Y_v(\sigma)
     }_Y \, d\sigma
   \\
    &\quad
     +\int_T^\infty
     \Norm{e^{\sigma A^-}\pi_-}
     _{\Ll(L^1,L^2)}
     \Norm{
     df|_{\eta(\sigma)}
     \circ Y_v(\sigma)
     }_Y \, d\sigma
   \\
    &\le
     \int_0^{{\frac{T}{2}}}
     ce^{-\sigma\mu}
     \Bigl(\kappa_*
     \underbrace{
      \Norm{\xi(\sigma)-\eta(\sigma)}_X
     }_{\text{$\le \frac{4}{7}\rho_0
        e^{-T\frac{\mu}{16}}$,~(\ref{eq:xi-eta})}}
     \Norm{X_v(\sigma)}_X
     +
     \kappa(\rho)
     \Norm{X_v(\sigma)-Y_v(\sigma)}_X
     \Bigr)
     d\sigma
   \\
    &\quad
     +\int_{{\frac{T}{2}}}^T
     ce^{-\sigma\mu}
     \kappa(\rho)
     \underbrace{
       \Norm{X_v(\sigma)}_X
     }_{\le 2c\sigma^{-\frac{1}{2}}
        e^{-\sigma\frac{\mu}{2}}
        \Norm{v}_2}
     \, d\sigma
     +\int_{{\frac{T}{2}}}^\infty
     ce^{-\sigma\mu}
     \kappa(\rho)
     \underbrace{
       \Norm{Y_v(\sigma)}_X
     }_{\le 2c\sigma^{-\frac{1}{2}}
        e^{-\sigma\frac{\mu}{2}}
        \Norm{v}_2}
     \, d\sigma
   \\
    &\le
     \frac{8}{7}\rho_0 c^2\kappa_*
     e^{-T\frac{\mu}{16}}\Norm{v}_2
     \int_0^{{\frac{T}{2}}}
     e^{-\sigma\frac{3}{2}\mu}
     \sigma^{-\frac{1}{2}}
     d\sigma
     +c\kappa(\rho)
     \Norm{X_v-Y_v}_*
     \int_0^{{\frac{T}{2}}}
     e^{-\sigma\mu}
     \sigma^{-\frac{1}{2}}
     d\sigma
   \\
    &\quad
     +
     4c^2\kappa(\rho)
     \Norm{v}_2
     \int_{{\frac{T}{2}}}^\infty
     e^{-\sigma\frac{3}{2}\mu}\sigma^{-\frac{1}{2}}
     d\sigma
   \\
    &\le
      \frac{1}{8}
      \Norm{X_v-Y_v}_*
     +
     \left(\frac{1}{4} e^{-T\frac{\mu}{16}}
     +\frac{1}{4} e^{-T\frac{\mu}{2}}\right)
     \Norm{v}_2.
\end{split}
\end{equation*}
Inequality two uses
that by the Lipschitz Lemma~\ref{le:f}
for $df$ and its Corollary~\ref{cor:f}
\begin{equation*}
\begin{split}
    &\Norm{
     df|_{\xi(\sigma)}
     \circ X_v(\sigma)
     -df|_{\eta(\sigma)}
     \circ Y_v(\sigma)
     }_Y
   \\
    &=
     \Norm{
     \left(df|_{\xi(\sigma)}
     -df|_{\eta(\sigma)}\right)
     \circ X_v(\sigma)
     +
     df|_{\eta(\sigma)}\circ 
     \left(X_v(\sigma)
     -Y_v(\sigma)\right)
     }_Y
   \\
    &\le
     \kappa_*
     \Norm{\xi(\sigma)-\eta(\sigma)}
     _X
     \Norm{X_v(\sigma)}_X
     +
     \kappa(\rho)
     \Norm{X_v(\sigma)-Y_v(\sigma)}_X
\end{split}
\end{equation*}
and
$
     \norm{df|_{\eta(\sigma)}
     \circ Y_v(\sigma)}_Y
     \le \kappa(\rho)
     \norm{Y_v(\sigma)}_X
$,
respectively. We treated
the integral over $[T/2,T]$
with the triangle inequality and incorporated
its $\eta$ part into the integral
over $[T/2,\infty)$.
Furthermore, use that
$\norm{\cdot}_{\Ll(L^1,L^2)}
\le\norm{\cdot}_{\Ll(L^1,X)}$
by the embedding 
$X\hookrightarrow L^2$,
then apply Proposition~\ref{prop:semigroup-NEW}.
Consider inequality three.
In the calculation above
we indicated how to estimate certain terms.
The estimates used are~(\ref{eq:xi-eta})
and~(\ref{eq:X_v-exp}). We also
used~(\ref{eq:X_v-exp}) for $Y_v$
with $T=\infty$.
In inequality four we
applied the estimate~(\ref{eq:Hoelder})
to deal with all integrals
and we used the smallness 
assumption~(\ref{eq:rho_0-backward}) on $\rho_0$
and~(\ref{eq:rho-backward}) on $\rho$.

It remains to prove exponential
decay of the weighted sup norm $\norm{\cdot}_*$
over the domain $[0,\frac{1}{2} T]$.
Fix $s\in[0,\frac{1}{2} T]$
and conclude similarly as above that
\begin{equation*}
\begin{split}
    &
     s^{\frac{1}{2}}
     \Norm{X_v(s)-Y_v(s)}_X
   \\
    &\le
     s^{\frac{1}{2}} c
     \int_0^s
     \frac{e^{-(s-\sigma)\mu}}
     {(s-\sigma)^{\frac{3}{4}}}
     \Bigl(\kappa_*
     \underbrace{
       \Norm{\xi(\sigma)
       -\eta(\sigma)}_X
     }_{\le\frac{4}{7}\rho_0e^{-T\frac{\mu}{16}}}
       \Norm{X_v(\sigma)}_X
     +\kappa(\rho)
       \Norm{X_v(\sigma)
       -Y_v(\sigma)}_X
     \Bigr) d\sigma
   \\
    &\quad
     +
     s^{\frac{1}{2}} c
     \int_s^{{\frac{T}{2}}}
     e^{(s-\sigma)\mu}
     \Bigl(\kappa_*
     \Norm{\xi(\sigma)
     -\eta(\sigma)}_X
     \Norm{X_v(\sigma)}_X
     +\kappa(\rho)
     \Norm{X_v(\sigma)
     -Y_v(\sigma)}_X
     \Bigr) d\sigma
   \\
    &\quad
     +
     s^{\frac{1}{2}}c\kappa(\rho)\biggl(
     \int_{\frac{T}{2}}^T
     e^{(s-\sigma)\mu}
     \underbrace{
       \Norm{X_v(\sigma)}_X
     }_{\le 2c\sigma^{-\frac{1}{2}}
            e^{-\sigma\frac{\mu}{2}}\norm{v}_2}
     d\sigma
     +
     \int_{\frac{T}{2}}^\infty
     e^{(s-\sigma)\mu}
     \underbrace{
       \Norm{Y_v(\sigma)}_X
     }_{\le 2c\sigma^{-\frac{1}{2}}
            e^{-\sigma\frac{\mu}{2}}\norm{v}_2}
     d\sigma\biggr)
   \\
    &\le
     \frac{8}{7}\rho_0 c^2\kappa_*
     e^{-T\frac{\mu}{16}}\Norm{v}_2
     \left(
     \int_0^s
     \frac{e^{-(s-\sigma)\mu}
     e^{-\sigma\frac{\mu}{2}} s^{\frac{1}{2}}}
     {(s-\sigma)^{\frac{3}{4}}
     \sigma^{\frac{1}{2}}}
     \, d\sigma
     +
     \int_s^{\frac{T}{2}}
     \frac{e^{(s-\sigma)\mu}
     e^{-\sigma\frac{\mu}{2}} s^{\frac{1}{2}}}
     {\sigma^{\frac{1}{2}}}
     \, d\sigma
     \right)
   \\
    &\quad
     +c\kappa(\rho)
     \Norm{X_v-Y_v}_*
     \left(
     \int_0^s
     \frac{e^{-(s-\sigma)\mu}
     s^{\frac{1}{2}}}
     {(s-\sigma)^{\frac{3}{4}}
     \sigma^{\frac{1}{2}}}
     \, d\sigma
     +
     \int_s^{\frac{T}{2}}
     \frac{e^{(s-\sigma)\mu}
     s^{\frac{1}{2}}}
     {\sigma^{\frac{1}{2}}}
     \, d\sigma
     \right)
   \\
    &\quad
     +4c^2\kappa(\rho)
     \Norm{v}_2
     \int_{{\frac{T}{2}}}^\infty
     e^{(s-\frac{3}{2}\sigma)\mu}
     \left(\frac{s}{\sigma}\right)
     ^{\frac{1}{2}}
     d\sigma
   \\
    &\le 
     \left(\frac{8}{7}\rho_0 c^2\kappa_*
     \left(\frac{10}{\mu^{1/4}}
     +\frac{1}{\mu}\right)
     +\frac{8 c^2\kappa(\rho)}{3\mu}\right)
     \frac{\Norm{v}_2}{e^{T\frac{\mu}{16}}}
     +c\kappa(\rho)\left(\frac{10}{\mu^{1/4}}
     +\frac{1}{\mu}\right)
     \Norm{X_v-Y_v}_*
   \\
    &\le 
     \Bigl(\frac{1}{4}+\frac{1}{4}\Bigr)
     \Norm{v}_2 e^{-T\frac{\mu}{16}}
     +\frac{1}{4}\Norm{X_v-Y_v}_*.
\end{split}
\end{equation*}
It is a side remark that without
the weight factor $s^{1/2}$ in the
$\norm{\cdot}_*$ norm the integrals involving
$(s-\sigma)^{-3/4}$ cause trouble,
concerning boundedness, for $s$ near zero. 
It is another side remark that
due to the presence of the extra factor 
$\norm{X_v(\sigma)}_X$ we do not
have to cut the interval $[\frac{T}{2},T]$
into two pieces as we did in step~6 above.
Inequality three uses the following estimates.
By~(\ref{eq:3/4}) and by calculation, 
respectively, we obtain
$$
     \int_0^s
     \frac{e^{-(s-\sigma)\mu}}
     {(s-\sigma)^{3/4}}
     \left(\frac{s}{\sigma}\right)^{1/2}
     d\sigma 
     \le 10/\mu^{\frac{1}{4}}
     ,\quad
     \int_s^{\frac{T}{2}}
     e^{(s-\sigma)\mu}
     \left(\frac{s}{\sigma}\right)^{1/2}
     d\sigma
     \le
     \frac{1}{\mu}.
$$
To get the second of these estimates
we used $s/\sigma\le1$.
Again by calculation we get
$$
     \int_{\frac{T}{2}}^\infty
     e^{(s-\frac{3}{2}\sigma)\mu}
     \left(\frac{s}{\sigma}\right)
     ^{\frac{1}{2}}
     d\sigma
     \le
     \int_{\frac{T}{2}}^\infty
     e^{(s-\frac{3}{2}\sigma)\mu}
     d\sigma
     \le
     \frac{2}{3\mu}
     e^{(s-\frac{3}{4}T)\mu}
     \le
     \frac{2}{3\mu}
     e^{-T\frac{\mu}{4}}
$$
since $s\le T/2\le\sigma$.
In the final inequality four use
smallness~(\ref{eq:rho_0-backward}) of $\rho_0$
and~(\ref{eq:rho-backward}) of $\rho$.
Now take the supremum over
$s\in[0,\frac{1}{2}T]$ to obtain
$
     \Norm{X_v-Y_v}_*
     \le
     e^{-T\mu/16}\Norm{v}_2
$.
Together with the estimate for
$\norm{X_v(0)-Y_v(0)}_2$
derived earlier this concludes the
proof of Step~II.
\qed\end{proof}
\noindent
This concludes the proof of Theorem~\ref{thm:uniform-C1}.

\begin{acknowledgements}
For hospitality I would like to thank Universit\"at Bielefeld where foundations were laid. In this respect I am most grateful to Helmut Hofer for the right words in a difficult moment. Many thanks to Andr\'{e} de Carvalho and Pedro Salom\~{a}o for building the bridge to a new continent and, in particular, the excellent research conditions provided by IME USP and FAPESP. Last, not least, the paper would not exist without Dietmar Salamon teaching me for many years his way of solving complex problems. I owe him deeply.
\end{acknowledgements}


\end{document}